\documentclass[12pt,leqno,twoside]{amsart}

\usepackage[all]{xy}

\usepackage{etex}
\usepackage{amscd}
\usepackage{amsmath}
\usepackage{amsthm}
\usepackage{amsfonts}
\usepackage{amssymb}
\usepackage{graphics}
\usepackage{epsfig}
\usepackage{epic}
\usepackage{eufrak}
\usepackage{bm}
\usepackage{pictex} % this is only necessary for the \setshagrid and related commands
\usepackage{color}

\usepackage{soul}
\usepackage{graphicx}

\newtheorem{thm}{\bf Theorem}[section]
\newtheorem{lem}[thm]{\bf Lemma}
\newtheorem{remark}[thm]{\bf Remark}
\newtheorem{example}[thm]{\bf Example}
\newtheorem{proposition}[thm]{\bf Proposition}
\newtheorem{cor}[thm]{\bf Corollary}
\newtheorem{dfn}[thm]{\bf Definition}

\def\bI{{\bf J}}

\def\bz{{\bf z}}
\def\b1{{\bf 1}}

\def\bB{{\bf B}}
\def\bT{{\bf T}}

\def\bP{{\bf P}}
\def\bQ{\ensuremath{\mathbb{Q}}}
\def\bR{\ensuremath{\mathbb{R}}}
\def\bk{{\bf k}}
\def\bM{{\bf M}}

\def\bN{{\bf N}}
\def\bZ{\ensuremath{\mathbb{Z}}}
\def\bC{\ensuremath{\mathbb{C}}}

\def\cM{\mathcal M}
\def\t{\mathsf t}

\newcommand{\coh}{\operatorname{coh}}

\newcommand{\C}{\ensuremath{\mathbb{C}}}

\newcommand{\mathP}{\ensuremath{\mathbb{P}}}

\newcommand{\R}{\ensuremath{\mathbb{R}}}

\newcommand{\Z}{\ensuremath{\mathbb{Z}}}

\newcommand{\mE}{\ensuremath{\mathcal{E}}}
\newcommand{\mF}{\ensuremath{\mathcal{F}}}
\newcommand{\mG}{\ensuremath{\mathcal{G}}}

\def\mO{{\mathcal O}}

\def\FA{{\mathfrak A}}
\def\Fb{{\mathfrak b}}
\def\FB{{\mathfrak B}}
\def\FX{{\mathfrak X}}

\newcommand{\Ext}{\mathop{\mathrm{Ext}}\nolimits}

\def\Blat{\mbox{\it \raise2pt\hbox{"}\kern-2pt H}}
\def\lvup{\rlap{\ ${}^{q\atop{\hbox{${}^{\vee}$}}}$}\cdots}

\def\cal{\mathcal}
\numberwithin{equation}{section} 
\begin{document}
\author{Susumu TANAB\'E}               % Here insert the authors name and the address
            % this must be left empty
\title [Period integrals associated to  an affine hypersurface]{Period integrals associated to  an affine Delsarte type hypersurface}               % Here insert the title
%\date{}               % this must be left empty
%\begin{document}

\begin{abstract}
	
	We calculate the period integrals for a special class of  affine hypersurfaces (deformed Delsarte hypersurfaces)
	in an algebraic torus  by the aid  of their Mellin transforms.
	A description of the relation between poles of Mellin transforms of period integrals and 
	the mixed Hodge
	structure of the cohomology of the
	hypersurface is given. By interpreting the period integrals as solutions to Pochhammer
	hypergeometric differential equation, we calculate concretely the irreducible monodromy 
	group of period integrals that correspond to the compactification of the affine hypersurface
	in a complete simplicial toric variety. As an application of the equivalence between oscillating integral for Delsarte polynomial and  quantum cohomology of a weighted projective space $\mathP_\bB$,
	we establish an equality between its Stokes matrix and the Gram matrix of the full exceptional collection on $\mathP_\bB$.
\end{abstract}
\maketitle
%%%%%%%%%%%%%%%%%%%%%%%%%%%%%%%%%%%%%%%%%%%%%%%%

%%%%%%%%%%%%%%%%%%

\noindent
%
%%%%%%%%%%%%%%%%%%%%%%%%%%%%%%%%%%%%%%%%%%%%%%%%%%%%%%%%%%%%%%%%%
{
\center{\section{Introduction}}
}

In this note, we propose a simple method to calculate concretely
period integrals associated to an  affine non-compact hypersurface 
for which the number of 
terms participating in its defining equation is larger than the dimension of the ambient algebraic torus  by two 
(deformed Delsarte hypersurface). A monomial deformation of a Fermat type polynomial belongs to this class. 
As for the historical reason of the naming, see Remark \ref{delsarte}.
As an important example of the polynomial under consideration, we point out the following Landau-Ginzburg potential
 familiar in the mirror symmetry setting \cite[4.2]{CK},
$$ f_0(x) = \sum_{j=1}^n x_j  + \prod_{j=1}^n \frac{1}{x_j}.$$

We establish an expression of the position of poles of the Mellin
transform with the aid of the mixed Hodge structure of cohomology groups associated to an
hypersurface $Z_f$ defined by a $\Delta-$regular polynomial \cite{Baty1} (See \S \ref{mellin}, Proposition \ref{prop31}). The quest to relate the
asymptotic behaviour of a period integral with the Hodge structure
of the algebraic variety goes back to \cite{Var} where Varchenko
established the equivalence of the asymptotic Hodge structure and
the mixed Hodge structure in the sense of Deligne-Steenbrink for
the case of plane curves  and (semi-)quasihomogneous
singularities.

\par
% Later on, several authors (\cite{LS},
%\cite{Sab1},\cite{Sab2},\cite{Sab3}) have pursued studies on the the asymptotic behaviour of period integrals in making use of the
%Mellin transforms. Their main idea consists in the fact, that it is possible to visualize the asymptotic behaviour (i.e. the filtration)
%of period integrals by means of the poles of Mellin transform. Especially in the case of
%complete intersection singularities, the advantage of this method is quite clear. 

%Let us remark also that not only poles of the Mellin transform but its zeros play role in the calculus of the global monodromy of the %period integrals(e.g. see Proposition ~\ref{prop53}).

%The relation between poles of the Mellin transform and the mixed
%Hodge structure has been explained for  examples of
%solated complete intersections of space curve type in \cite{Tan}.
 In this note, we illustrate the utility of this approach
in taking the example of a hypersurface in a torus defined by
so called simpliciable polynomial
(see Definition ~\ref{dfn1}). 
 %\vspace{1.5pc}
\footnoterule

\footnotesize{AMS Subject Classification: 14M109 (primary), 32S25, 32S40
(secondary).

Key words and phrases: affine hypersurface,
Hodge structure, hypergeometric function.}

 \footnotesize{Partially supported by Max Planck Institut f\"ur Mathematik, T\"UBITAK 1001 Grant No. 116F130 "Period integrals associated to algebraic varieties."}
\normalsize

%%%%%%%%%%%%%%%%%%%%%%%%%%%%%%%%%%%%%%%%%%%%%%%%%%%%%%%%%%%%%%%%%

%%%%%%%%%%%%%%%%%%%%%%%%%%%% FOOTNOTES %%%%%%%%%%%%%%%%%%%%%%%%%%%%%%%

\newpage

For this class of hypersurfaces, we  can present period integrals as solutions to so called Pochhammer hypergeometric equation
(\ref{RkI}) by using their Mellin transform. The solution space to this equation has reducible monodromy, but it is possible to subtract 
a  solution subspace with irreducible monodromy (Theorem \ref{thm53}) that corresponds to period integrals of the compactified quasi-smooth hypersurface $\bar Z_f$ in a complete simplicial toric variety ${\mathbb P}_{\Delta(f)}$ (\cite[Definition 3.1]{BatyCox}).

 In \cite[Theorem 14.2]{Baty1} , \cite[Theorem 8]{Stienstra}, \cite[3.4]{KM}, authors gave interpretations of  period integrals of affine hypersurfaces as A-hypergeometric functions that form a vector space with dimension equal to the volume of the Newton polyhedron of $f$ . They did not, however, discuss the reducibility/irreducibility of the global monodromy. 
In fact our restriction on number of terms of the polynomial $f$ is motivated by the fact that in this setting the A- HG system is reduced an univariable  Pochhammer HG equation whose monodromy can be concretely calculated (see Theorem \ref{thm53}). If we consider a Laurent polynomial with more terms than considered here, it is necessary to  deal with a multivariable  holonomic system that makes the calculation of monodromy far more difficult as it would presume the monodromy as a representation of the fundamental group of the holomorphic domain of A-HGF ( \cite{Tan041}, \cite{Tan07}, \cite{Ta17}) .

The major technical novelty of this note, in comparison with the previous works 
concerning period integrals in question ( \cite{Salerno}, \cite{Dominguez}, \cite{Gahrs}), lies in the full use of their Mellin-Barnes integral representation
instead of Gauss-Manin system or Griffiths-Dwork method (Corollary \ref{cor:IukxI}). See Remark \ref{salerno}. This integral representation is nothing but a special case of general formula of Horn type HG equations in several variables satisfied by period integrals  \cite{Tan07}.
This method has been already used in \cite{Tan04},\cite{TaU13}. It  essentially relies on a combination of simple topological argument on cycles,  definition of the $\Gamma-$function and expression of Dirac $\delta-$type function by a Laplace transform. 
In Corollary \ref{cor:IukxI}, an explicit description of the dependence on the integrand differential form (a representative of the cohomology) allows us to establish a comparison result Lemma \ref{LqtildeLq} on filtrations defined by Mellin transform, from one side, and by Brieskorn lattice, from the other \cite{DS}.

The relation between the reducible monodromy group  of period integrals and the Stokes matrix of corresponding oscillating integrals
has been discussed in \cite[Theorem 1.1, 1.2]{Tan04}, \cite[Theorem 5.1]{TaU13}.  As it is shown in \S \ref{oscillating} of this note, Batyrev's quotient ring $R_f^+$ (\ref{RF+}) is well adapted to the description of  the space of  oscillating integrals 
(\ref{JgGamma}).

In \cite[ 4.2]{Horja}, relying on the Stanley-Reisner ring method \cite[Theorem 6]{Stienstra},  the author subtracts period integrals of compact complete intersection from those of affine complete intersection treated in  \cite{TaU13}. R.P. Horja discusses the analytic continuation of these period integrals  and confirms  a correspondence  predicted  by Kontsevich \cite{Kontsevich_HAMS} in connection with homological mirror symmetry conjecture.  None the less he did not give a complete description of (irreducible) global monodromy group of period integrals. Our matrices (\ref{hinf}), (\ref{h0}) in Lemma \ref{levelt} (irreducible monodromy) and
 (\ref{eq:hinfty}),  (\ref{eq:h0}) (reducible monodromy) give a global monodromy representation of period integrals.

In \cite{CG}, the authors  studied the irreducible monodromy acting on the structure sheaf of an affine complete intersection treated in  \cite{TaU13} that they subtracted  from 
the vector space with reducible monodromy action. This subtraction procedure to get an irreducible monodromy representation has been discussed in \cite{Gol} before. Our note has genetic similarity with  
\cite{CG}, even though the methods used are quite different. We use the Mellin transform of period integrals, while Corti-Golyshev
relied uniquely on \cite{DX1} to prove their main theorems \cite[Theorem 1.1, 1.3]{CG} 
that correspond to our Proposition \ref{prop61}, Remark \ref{WHnZ}, Theorem \ref{hln}.

The contents of this note can be summarised as follows.
In \S 1, we give a review on the mixed Hodge structure of the cohomology associated to an affine hypersurface according to \cite{Baty1}. In \S \ref{simplicial} under the main {\bf Assumption} imposed on $f(x)$, principal integral data $\gamma$ (\ref{gamma}) and $\bB$ (\ref{sumBq1}) are introduced. 
In \S \ref{mellin}, we establish Proposition \ref{prop31} that calculates concretely the Mellin transform of period integrals.
In \S \ref{oscillating}, we write down a differential equation with irregular singularities satisfied by oscillating integrals
associated to a Laurent  polynomial (or Landau-Ginzburg potential) $f_0.$
In \S \ref{filtration}, we examine the relation between a Mellin transform based filtration of period integrals and Kashiwara-Malgrange filtration on the Brieskorn lattice by \cite{DS}. In the core part of the note \S \ref{HGgroup}, a concrete representation of monodromy group for period integrals in terms of Pochhammer HGF and its Hermitian invariant are described. In \S \ref{WHP}, we discuss Dubrovin's conjecture \cite{Dub2} on the Stokes matrix for the quantum cohomology of the weighted projective space $\mathP_{\bB}$ treated in  \S \ref{oscillating}.
\par
I would like to express my gratitude to Profs. M.Oka, K.Saito, K. Takeuchi, K.Ueda, D.Cox, J.Tevelev, E.Materov, F.Beukers who gave me various occasions to discuss on this subject
and to formulate this hypersurface version of the calculus. I also express special gratitude to Dr. Elif Segah \"Ozta\c{s} who prepared scrupulously the figure for Example \ref{fig:quotientring}. The complete intersection version in preparation \cite{Tan041} shall describe the Kashiwara-Malgrange filtration of period integrals by the aid of Mellin transform in several variables.
Special thanks go to Prof. T. Terasoma who indicated the possibility to read off the mixed Hodge structure of cohomology from the Mellin transform of period integrals. It is worth noticing that I got many valuable suggestions from his inspiring article \cite{Terasoma}.

{
\center{\section{Hodge structure of the cohomology group of a hypersurface
in a torus}\label{Hodge}}
}
In this section, we review fundamental notions on
the Hodge structure of the cohomology group of a hypersurface
in a torus after \cite{Baty1}, \cite{DX1}.

\par Let $\Delta$ be a convex $n-$dimensional convex polyhedron in
${\bR}^n$ with all vertices in $\bZ^n.$
Let us define a ring $S_\Delta \subset$
$\bC[u, x^{\pm}] := \bC[u,x_1^{\pm}, \cdots, x_n^{\pm}]$ of
the Laurent polynomial ring as follows
\begin{equation}
S_\Delta:= \bC \oplus \bigoplus_{\frac{ \alpha}{k} \in \Delta,
\exists k \geq 1} \bC \cdot u^kx^{ \alpha}.
\label{Sdelta}
\end{equation}
%\leqno(1.1) 

\par
We denote by $\bar \Delta(f)$ the convex hull of $\{0\} \in \bZ^{n+1}$ and the set $\{(1,{ \alpha}) \in \bZ^{n+1};
{ \alpha} \in supp(f) \}$ that we call the Newton polyhedron of a Laurent
polynomial  (further simply called polynomial) $F(u,x)= uf(x)-1.$  For $\bar \Delta(f)$,  we introduce the following Jacobi ideal
\begin{equation}\label{jideal}
 J_{f,\Delta}= \bigl<\theta_u F, \theta_{x_1}F,
\cdots,   \theta_{x_n}F\bigr>\cdot S_{ \Delta(f)}.
\end{equation}
%\leqno(1.2)$$
with $ \theta_{u}=u \frac{\partial }{\partial u} $, $\theta_{x_i}=x_i \frac{\partial }{\partial x_i}, i\in [1;n]$.

From here on, we shall further use the notation $i \in [m_1; m_2] \Leftrightarrow i \in \{m_1, \cdots,  m_2\} $ for two integers $m_1 <m_2.$

Let $\tau$ be a $\ell-$dimensional face of $\Delta(f)$, the convex hull of $\{{ \alpha} \in \bZ^{n};
{ \alpha} \in supp(f) \}$ in $\bR^n$,
and define
\begin{equation}
 f_\tau(x)= \sum_{ \alpha \in \tau \cap supp(f)}a_{ \alpha}
x^{ \alpha},
\end{equation}
%\leqno(1.3)$$
where $f(x) =\sum_{ \alpha \in supp(f)}a_{ \alpha}
x^{ \alpha}.$
The Laurent polynomial $f(x)$ is called {\it $\Delta-$ regular}, if
$\Delta(f)=\Delta$ and for every $\ell-$dimensional face $\tau
 \subset \Delta(f)$ ($\ell >0$) the polynomial equations:
$$ f_\tau(x)= \theta_{x_1} { f_\tau}=
\cdots=  \theta_{x_n} { f_\tau}=0,$$
have no common solutions in $\bT^n = (\bC^\times)^n.$

\begin{proposition} (\cite[Theorem 4.8]{Baty1}.)
Let $f$ be a Laurent polynomial such that $\Delta(f)=\Delta.$
Then the following conditions are equivalent.
\par
1) The elements $uf, u\theta_{x_1}  f,
\cdots,  u\theta_{x_n}  f $ give rise to a regular sequence
in $S_{\Delta}$
\par
2) For the Jacobi ideal  $J_{f,\Delta}$  \eqref{jideal},
the following equality holds 
$$ dim \bigl(\frac{ S_\Delta}{J_{f,\Delta}}\bigr)= n! vol(\Delta).$$
\par
3) $f$ is $\Delta-$regular.
\label{prop: laurent}
\end{proposition}

For a $\Delta-$regular polynomial $f$,  we shall further denote the space $\frac{ S_\Delta}{J_{f,\Delta}}$ by $R_f,$
\begin{equation}
 R_f= \frac{ S_\Delta}{J_{f,\Delta}}. 
\label{Rf}
\end{equation}
%\leqno(1.2)'$$
It is possible to  introduce a filtration on $S_\Delta,$
namely $u^i x^{ \alpha} \in S_k $ if and only if
$i \leq k$ and $ \frac{ \alpha}{k} \in \Delta.$
Consequently, we have an increasing filtration;
$$\bC \cong \{0\}=S_0 \subset S_1 \subset \cdots \subset S_n \subset \cdots,$$
that induces a decreasing filtration on  $R_f$ so that the decomposition $$R_f = \bigoplus_{i=0}^n R_f^i$$ holds where $R_f^i$
the $i-$the homogeneous part of $R_f$.

It is worthy to remark here that the filtration on $R_f$ ends up with the $n-$th
term.
\par
Let us recall the notion of Ehrhart polynomial:
\begin{dfn}
Let $\Delta$ be an $n-$dimensional convex polytope. Denote the
Poincar\'e series of graded algebra $S_\Delta$ by
$$P_\Delta({\mathsf t})= \sum_{k \geq0} \ell(k\Delta){\mathsf t}^k,$$
$$Q_\Delta({\mathsf t})= \sum_{k \geq0} \ell^\ast(k\Delta){\mathsf t}^k,$$
where $\ell(k\Delta)$ (resp.$\ell^\ast(k\Delta)$ )
represents the number of integer points
in $k\Delta.$ (resp. interior integer points in $k\Delta.$ )
Then
$$\Psi_\Delta({\mathsf t})= \sum_{k = 0}^n \psi_k(\Delta){\mathsf t}^k= (1-{\mathsf t})^{n+1}P_\Delta({\mathsf t}),$$
$$\Phi_\Delta({\mathsf t})= \sum_{k = 0}^n \varphi_k(\Delta){\mathsf t}^k = (1-{\mathsf t})^{n+1}Q_\Delta({\mathsf t}),$$
are called Ehrhart polynomials which satisfy
$${\mathsf t}^{n+1}\Psi_\Delta({\mathsf t}^{-1})= \Phi_\Delta({\mathsf t}).$$
\label{Ehrhart}
\end{dfn}
Let $\bT^n=(\bC\setminus \{0\})^n$ $=$ $Spec$ $\bC[x_1^{\pm}, \cdots, x_n^{\pm}]$ 
and $\bT^{n+1}=(\bC\setminus \{0\})^{n+1}$ $= Spec \bC[u^{\pm},x_1^{\pm},$ $ \cdots, x_n^{\pm}].$
Further, the main object of our study will be the cohomology group
of the complement to the hypersurface $Z_f:=\{x \in \bT^n ; f(x)=0\},$ i.e.
$\bT^n \setminus Z_f$ $\cong \{(u,x) \in \bT^{n+1} ; -uf(x)+1=0\}.$
The primitive part
$PH^{n}(\bT^n \setminus Z_f)$ is defined by the following exact sequence
\begin{equation}
0 \rightarrow H^n(\bT^{n+1}) \rightarrow H^n(
Z_{-uf+1}) {\rightarrow} PH^{n}( Z_{-uf+1})
\rightarrow 0.
\label{PHn}
\end{equation}
We consider its Hodge filtration 
$$ 0=F^{n+1}PH^{n}(\bT^n \setminus Z_f)  \subset \cdots \subset F^1PH^{n}(\bT^n \setminus Z_f)=
F^0 PH^{n}(\bT^n \setminus Z_f) = PH^{n}(\bT^n \setminus Z_f).$$
\begin{thm}(\cite[Theorem 6.9]{Baty1})
For the primitive part $PH^{n}(\bT^n \setminus Z_f)$ of $H^{n}(\bT^n \setminus Z_f),$
the following isomorphism holds
\begin{equation}
\frac{F^iPH^{n}(\bT^n \setminus Z_f)}{F^{i+1}PH^{n}(\bT^n \setminus Z_f)} \cong
R_f^{n+1-i}. \;\;\;  i \in [1;n+1 ]
\label{Rfni}
\end{equation}
%\leqno(1.4)$$
Furthermore
$$ dim R_f^{n+1-i}=
%\sum_{q \geq 0}h^{i,q}(PH^{n-1}(Z_f))
\psi_{n+1-i}(\Delta),$$
for $ i \leq n.$
\label{thm12}
\end{thm}

Denote by $I_\Delta^{(\ell)}$ $(\ell \in [0; n+1])$ the homogeneous ideal of $S_\Delta$
generated as a $\bC$ vector space by all monomials  $u^kx^{ \alpha}$ such 
that $\frac{ \alpha}{k}$ is located in $\Delta$ but not on any face 
 $\Delta' \subset \Delta$ with codimension $\ell.$

Thus we obtain the increasing chain of homogeneous ideals in $S_\Delta$
\begin{equation}
0=I_{\Delta}^{(0)}\subset I_{\Delta}^{(1)} \subset \cdots \subset I_{\Delta}^{(n)}  \subset I_{\Delta}^{(n+1)}=S_\Delta^+, 
\end{equation}
%\leqno(1.5)$$
where $S_\Delta^{+}$ is the maximal homogeneous ideal in $S_\Delta.$

Denote by ${\mathcal R }$ the $\bC$ linear mapping
$$ {\mathcal R} : S_\Delta \rightarrow \Omega^n(\bT^n \setminus Z_f)$$
defined as
\begin{equation}
   {\mathcal R}(u^kx^{ \alpha}) = \frac{(-1)^k (k-1)! x^{ \alpha}}{f(x)^{k} } \frac{dx}{x^\b1} 
\label{Rukxa}
\end{equation}
%\leqno(1.6)$$
with $dx =  \bigwedge_{i=1}^n dx_i$, $x^\b1 = \prod_{i=1}^n x_i.$ Further we shall also use the notation $\omega_0 =  \frac{dx}{x^\b1}.$ We recall that in \cite[\S 10 Modified Laplace transform]{Dwork} the correspondence \eqref{Rukxa} was substantial. 

We introduce a decreasing $\cal E$- filtration on $S_\Delta$
$$ {\cal E}: \cdots \supset {\cal E}^{-k} \supset \cdots \supset   {\cal E}^{-1} \supset {\cal E}^{0}$$
where ${\cal E}^{-k}$ denotes the subspace spanned by monomials $u^\ell x^\alpha \in S_\Delta$
with $\ell \leq k.$

\begin{thm} (\cite[Theorem 7.13, 8.2] {Baty1},  \cite[Theorem 7] {Stienstra},  \cite[Theorem 4.1, 4.2] {KM})

1)  There exists the following commutative diagram
\begin{displaymath}
    \xymatrix{
        S_\Delta \ar[r] \ar[d]_{\mathcal R} & \frac{S_\Delta }{{\mathcal D}_u S_\Delta + \sum_{i=1}^n{\mathcal D}_{x_i}  S_\Delta} \ar[d]^{\rho} \\
        \Omega^n(\bT^n \setminus Z_f) \ar[r]       &  H^{n}(\bT^n \setminus Z_f) }
\end{displaymath}
with ${\mathcal D}_u (g) =  e^{uf}\theta_u (e^{-uf}g)$,  ${\mathcal D}_{x_i} (g) =  e^{uf}\theta_{x_i} (e^{-uf}g)$ and $\rho$ an isomorphism.
In particular, we have the following isomorphism 
\begin{equation}
\rho^+: R_f^+ \rightarrow  PH^{n}(\bT^{n} \setminus Z_f), 
\label{rhoplus}
\end{equation}
for 
\begin{equation}
R_f^+=\frac{S_\Delta^{+}}{{\mathcal D}_u S_\Delta+ \sum_{i=1}^n{\mathcal D}_{x_i}  S_\Delta },
\label{RF+}  
\end{equation}
such that 
\begin{equation}
R_f \cong \bC \oplus R_f^+
\label{RfRF}.
\end{equation}

2)
The weight filtration on $H^{n}(\bT^{n} \setminus Z_f)$ is given by a increasing
filtration
\begin{equation}
0=W_{n}\subset W_{n+1} \subset \cdots \subset W_{2n} =H^{n}(\bT^{n} \setminus Z_f).  
\end{equation}
%\leqno(1.7)$$
For $1 \leq i \leq n-1,$ the subspace $W_{n+i}$ equals $\rho({\mathcal I}^{(i)})$
where ${\mathcal I}^{(i)}$ is the image of the ideal $I^{(i)}_\Delta$ in the space 
(\ref{RF+} ).
%While the remaining cases are described by $\rho({\mathcal I}^{(n+1)})  = W_{2n-2} = W_{2n-1}$,
%$ \rho({\mathcal I}^{(n+2)}) = H^{n}(\bT^{n} \setminus Z_f). $

3) 	The graded quotient of $R_f^+$ with respect to the  $\cal E$- filtration  is given by 
\begin{equation}
Gr^i_{\cal E} R_f^+= {\cal E}^{-i}( R_f^+ )/{\cal E}^{-i+1} (R_f^+) =   R_f^i .
\label{GrE}  
\end{equation}
for $ i \in [0; n].$ In particular $  R_f^+ = R_f^+ \cap {\cal E}^{-n}.$
%\leqno(1.8)$$
\label{thm13}
\end{thm}

The exact sequence 
%\begin{displaymath}\xymatrix{ 
$$0 \rightarrow H^n(\bT^n) \rightarrow  H^n(Z_{-uf+1}) \rightarrow^{res} H^{n-1} (Z_f) \rightarrow 0 $$
%\end{displaymath}
gives rise to the isomorphisms 
$$ res (F^i PH^n (Z_{-uf+1})) = F^{i-1}PH^{n- 1}(Z_{f}),  \;\;  i \in [1; n+1] $$
and
\begin{equation} res (W_j PH^n (Z_{-uf+1})) = W_{j-2}PH^{n- 1}(Z_{f}), \;\;  j \in [n+1; 2n]
\label{WjPH}
\end{equation}
 (\cite[Proposition 5.3] {Baty1}).

%%%%%%%%%%%%%%%%%%%%%%%%%%%%%%%%%%%%%%%%%%%%%%%%%%%%%%%%%%%%%%%%%
{
\center{\section{Simpliciable polynomial} \label{simplicial}}
}
Let us consider a  Laurent polynomial satisfying conditions of Proposition \ref{prop: laurent}
\begin{equation}
F(x)= \sum _{ i \in [1; n+2]} a_{i}x^{\alpha(i)}. 
\label{Fpoly}
\end{equation}
%\leqno(2.1)$$
Here $\alpha(i)$ denotes the multi-index
$$\alpha(i)=(\alpha^i_1, \cdots, \alpha^i_n ) \in \bM $$
for an integer lattice $\bM \cong {\bf Z}^{n}.$
Further we impose the following conditions on the polynomial $F(x)$

\vspace{0.5cm}
{\bf Assumption}
The point  ${\alpha(n+2)} \in \bM$ is located in the interior of the convex hull of $\{\alpha(i)\}_{i=1}^{n+1}$ that is an $n-$dimensional simplex.

\vspace{0.5cm}
This assumption means that the interior point fan $\Sigma$ defined by the Newton polyhedron $\Delta(F)$
is a complete simplicial fan. 
The defining equation of an  affine variety  $ Z_F$ defined in a torus $\bT^{n}$
is determined up to multiplication by a monomial $x^m,$ $m \in \bM.$ Therefore one can always assume one of terms participating in the expression (\ref{Fpoly}) to be a constant.
The  convention ${\alpha(n+1)} =0  \in \bM$  fixes the index of the constant term.

The main  object of our further study  is the polynomial 
$f(x) \in \bC[x^{\pm 1}][s]$ depending on a  parameter $s \in \bC,$
\begin{equation}
f(x)=\sum_{j=1}^{n} x^{ \alpha(j)}+1+s x^{ \alpha(n+2)}.  
\label{fpoly}
\end{equation}
%\leqno(2.2)$$
Further we use the convention $  \alpha(n+1) =0.$
To recover the situation in \S \ref{Hodge}, we need to put
\begin{equation}
f_0(x)=x^{- \alpha(n+2)}f(x)-s.
\label{f0}
\end{equation}
Indeed  a polynomial $ u F(x) $ with non-zero coefficients
can be reduced to the form $u f(x)$ by the aid of a torus $\bT^{n+1}$ action on the variables $(x,u).$

\begin{remark}{\rm
A polynomial that depends on $n-$variables and contains
$n$ monomials is called of Delsarte type. Jean Delsarte
established a formula counting points over a finite field  on the hypersurface defined by
a polynomial of this class \cite{Delsarte}.  T.Shioda found an algorithm to calculate explicitly the Picard number of this kind of surface ($n=3$) \cite{Shioda}. Delsarte surface began to draw attention of geometers in connection with the mirror symmetry conjecture and detailed studies of its  N\'eron-Severi lattice.
 One can consider $Z_f$ defined for (\ref{fpoly}) as a one dimensional deformation of a Delsarte type hypersurface.}
\label{delsarte}
\end{remark}

Let us introduce new variables $T_1, \cdots, T_{n+2}$:
$$T_j = u  x^{\alpha (j)},\; j \in [1; n], $$ 
\begin{equation} T_{n+2}=usx^{\alpha (n+2)}, T_{n+1} =u.  
\label{Tn1}
\end{equation}
%\leqno(2.4) $$

In making use of these notations , we have the relation
\begin{equation}
\begin{array}{c} 
 log\; T_j = log\;u + <\alpha(j), log\; x > ,\;  j \in [1;  n],\\ 
log \; T_{n+1}  = log\; u, \;\; log\; T_{n+2}= log\;u + <\alpha(n+2), log\; x >  +log\;s. \\
Log\; \Xi := ^t(\log\; x_1, \cdots, \log\; x_n,, \log\;s, \log\;u). \\
\end{array}
\label{Tnlog}
\end{equation}
%\leqno(2.7)$$

We can rewrite the relation (\ref{Tnlog}) with the aid of a matrix
${\sf L} \in End({\bf Z}^{n+2}),$
as follows
\begin{equation} Log\; T= {\sf L}\cdot Log\; \Xi . 
\label{TNlogs}
\end{equation} %\leqno(2.9)$$
where
\begin{equation}
{\sf L}= \left [\begin {array}{cccccccccc} {\it
\alpha_1^{1}}& \cdots&\alpha_n^{1} &0&1\\
\vdots&\cdots&\vdots&\vdots&1\\
\\\noalign{\medskip}\alpha_1^{n}&\cdots &\alpha_n^{n}&0&1\\
0&\cdots &0&0&1\\
\\\noalign{\medskip}\alpha_1^{n+2}&\cdots &\alpha_n^{n+2}&1&1\\

\end {array}\right ].
\label{L}
\end{equation}
%\leqno(2.10)$$

 We denote the determinant of the matrix (\ref{L})
by 
\begin{equation}
\gamma = det(\sf L).
\label{gamma}
\end{equation}
We remark here that a map similar to (\ref{Tnlog}), (\ref{TNlogs}) has been 
introduced in the proof of the main theorem in \cite{Shioda} where a relation of Delsarte surfaces 
to  Fermat surfaces is established.

\begin{dfn}
We call  a polynomial $f(x)$ simpliciable if 
%there exists
%$ \in {\mathsf S}_M$ such that 
$det(\sf L)=\gamma \not =0.$
\label{dfn1}
\end{dfn}

A Laurent polynomial $f(x)$ is simpliciable if and only if its Newton polyhedron $\Delta(f)$ has the dimension of the ambient torus $\bT^n$ that is equal to $n.$  A polynomial $F(x)$ satisfying the above {\bf Assumption} is simpliciable. 
Further we shall assume that the determinant $\gamma$ of the matrix ${\sf L}$
is positive for a simpliciable $f(x)$ in such a way that $\gamma  = n! vol(\Delta(f))$. This assumption is always satisfied without loss of generality,
if we permute certain row vectors of the matrix, which evidently corresponds
to the change of names of vertices $\alpha(j).$ 

\begin{lem} Let $f(x)$ be a simpliciable polynomial.
For the simplex polyhedron
$\tau_q \in \bR^{n}$ defined as $\bigl<    \alpha (1), \lvup , \alpha (n+1),  \alpha (n+2)  \bigr>,$  $  q \in [1, n],$ we introduce the following positive integer 
\begin{equation}
 B_q = n !vol(\tau_q). 
\label{Bq}
\end{equation}
%\leqno(3.8)$$
The relation below holds for \eqref{Bq},
\begin{equation}
 \sum_{q=1}^{n+1} B_q = |\bB| =\gamma = (-1)^{n+1} \chi(Z_f), 
\label{sumBq}
\end{equation}
%\leqno(3.9)$$
where $\chi(Z_f)$ denotes the Euler-Poincar\'e characteristic of the affine hypersurface $Z_f.$ In the sequel, we shall use the notation
\begin{equation}
 \bB =( B_1, \cdots, B_{n+1}). 
\label{sumBq1}
\end{equation}
\label{Bvector}
\end{lem}
\begin{proof}
The derivation of positive integers $B_1, \cdots, B_{n+1}$ is based on the calculation of $n+1 $ minors of  the matrix $\sf L$ obtained in  removing the $(n+2)-$nd column. 
To establish the last equality, we recall Theorem 2 of \cite{X} or Theorem 1 of \cite{Oka}
on the Euler characteristic and the volume of Newton polyhedron. 

\end{proof}

\begin{remark}
If a polynomial $f(x)$ is simpliciable, it has $n$ tuple of  linearly independent vectors from
$ supp(f).$ Such a polynomial with generic coefficients is $\Delta(f)$-regular.

%On the other hand for a non-simpliciable polynomial one can find $k-$tuple of vectors from $supp(f)$ whose linear span has %dimension strictly less than $k \geq n+2.$  It is easy to construct an example of $\Delta(f)$ regular polynomial satisfying such a %condition.
\label{remark:nonsimpliciable} 
\end{remark}

%%%%%%%%%%%%%%%%%%%%%%%%%%%%%%%%%%%%%%%%%%%%%%%%%%%%%%%%%%%%%%%%%
{
\center{\section{Mellin transforms}\label{mellin}}
}
In this section, we proceed to the calculation of  the Mellin transform of
period integrals associated to the hypersurface
$Z_{f} =\{ x \in \bT^{n} ; f(x)=0\}$
defined by a simpliciable
polynomial  $f$   (\ref{fpoly}). 

First of all, we consider the period integral taken along the fibre
%${ \delta(s)}$ 
for $u^k x^{\bI} \in R_f^+$
$\rho^+ (u^k x^{\bI}) \in PH^n(\bT^n \setminus Z_{f})$  (see Theorem  \ref{thm13} ) as follows
\begin{equation} I_{u^k x^{\bI}, \t \delta} (s): =
%2\pi \sqrt -1 \int_{ \delta(s)}\frac{x^{\bI} \omega_0 }{df(x)}=
%\left(\frac{1}{2\pi \sqrt -1}\right)^n
\int_{\t \delta(s)}\frac{  (k-1) ! x^{\bI} \omega_0 }{f(x)^k}      
\label{IukxI}
\end{equation}
%\leqno(3.1)$$
where ${\t \delta(s)}$
$\in H_{n}(\bT^{n} \setminus Z_{f})$ is a cycle
obtained after the application of  $\t:$ Leray's coboundary (or tube)  operator to a $n-1$ cycle
$ \delta(s)$
$\in H_{n-1}(Z_{f}).$ 
%Here $X^\b1 = X_1 \cdots X_{M-1},$ $X^{\bI}= X_1^{i_1} \cdots,
% X_{M-1}^{i_{M-1}}.$ 
%See the book by V.A.Vassiliev on ramified integrals  for the 
Leray's coboundary operator can be defined as a $S^1$ bundle construction over the cycle $\delta(s)$ (\cite{FFLP}, Part II).
%(\hl{ book by V.A.Vassiliev if the homology with non-comapct support is described. Otherwise cite Fotiadi-Lascoux- Pham})

The Mellin transform of
$I_{u^k x^{\bI}, \t \delta} (s)$ is defined by the following integral:
\begin{equation}M_{u^k x^{\bI}, \delta} (z):=  \int_{\Pi} s^z I_{u^kx^{\bI}, \t \delta } (s) \frac{ds}{s}. 
\label{MukxI}
\end{equation}
%\leqno(3.2)$$ 
Here $\Pi$ stands for a semi-real axis of the form $\Pi= \{s \in \bT; Arg\; s = \alpha\}$
for some fixed $\alpha \in [0, 2 \pi)$
that avoids ramification loci of $I_{u^k x^{\bI}, \t \delta} (s).$
 %We assume that on the set $\partial \gamma^\Pi:= \cup_{s \in \Pi}
%(\t(\gamma)(s),s),$ $\Re (f(X)+s) >0$ possibly except compact part of $\partial \gamma^\Pi.$

Here we recall the fact
$$ \int_{\bR_+} u^k e^{-u f(x)} \frac{du}{u}  = \frac{(k-1)!}{   f(x)^k}$$
for $ \Re(f(x))>0,$ $k \geq 1$. On the Leray coboundary $\t \delta(s) \subset \{x \in \bT^n; \mid f(x)\mid = \epsilon, \epsilon > 0\}$ 
the argument $Arg\; (f(x))$ moves on the circle $S^1$. Thus we introduce a fibre product  along $S^1$
$$ \bT \times_{S^1} \t  \delta(s) : = \bigcup_{\theta \in S^1} \{ (u,x)  \in (e^{-i \theta} \bR_+, \t \delta(s));   f(x) = \epsilon e^{i \theta}\}$$
in order to define the integral 
$$
\int_{ \bT \times_{S^1} \t  \delta (s)} e^{-u f(x) } x^{\bI}  du \wedge \omega_0 $$
properly as a function in $s \in \bT$. In fact the integrand function has neither branching points nor poles on the $\bC_u$ plane and, in general, the turn of the integration path 
$e^{-i \theta} \bR_+$ gives a natural analytic continuation beween integrals
$\int_{\bR_+} e^{-uT} du$ for $T>0 $ and $\int_{ e^{-i \theta} \bR_+} e^{-uT} du$ for 
$Arg\; T \in ( -\pi/2+\theta, \pi/2+ \theta).$
Now we consider the following $(n+2)-$ dimensional chain 
\begin{equation}
 \tilde \Gamma  : = (\bT \times_{S^1} \t  \delta (s)) \times_{\Pi} \Pi = \{(u,x,s); (u,x) \in \bT \times_{S^1} \t \delta (s), s \in \Pi \}.\label{tildeGamma}
\end{equation}
In fact,  this gives an equivariant fibration over $ {S^1} \times \Pi.$ The movement of   $(\theta, s)$ inside ${S^1} \times \Pi$ provokes no monodromy of the fibre. The above procedure can be considered as a justification of a formal argument depicted in \cite[\S 10]{Dwork}.

We deform the integral (\ref{MukxI}),  in making use of the relation
(\ref{TNlogs})
\begin{equation}
M_{u^k x^{\bI},  \tilde \Gamma} (z)=
\int_{ \tilde \Gamma} e^{-u f(x) } x^{\bI} u^{k}s^{z} \frac{du}{u} \wedge \omega_0 
\wedge \frac{ds}{s}
\label{ MukxIz}
\end{equation}
%\leqno(3.3) $$
$$= \frac{1}{\gamma}\int_{{\sf L}_\ast
(\tilde \Gamma)} e^{-\Psi(T)}
\prod_{q=1}^{n+2} T_q^{{\mathcal L}_q(\bI,z,k)} \prod_{q=1}^{n+2} \bigwedge
\frac{dT_q}{T_q},$$ with
\begin{equation}
\Psi(T) = T_1(x,u) + \cdots +T_{n+1}(u)+ T_{n+2}(x,s,u)= u f(x), 
\label{Psi} \end{equation}
%\leqno(3.4)$$
where each term $T_i(x,u), $ $ i \in [1; n],$ represents a
monomial term (\ref{Tn1}) of variables $(x,u)$  of the polynomial (\ref{Psi}) while
$T_{n+1}(u)=u.$
 Here the exponents ${\mathcal L}_q(\bI,z,k)$ denote  linear functions of components that shall be concretely given in 
 (\ref{Lq} ).

In the following proposition we denote by ${\mathcal L}_q(\bI,z,k)$ the inner product of $(\bI,z,k)$
with the $q-$th column vector of 
${\sf L}^{-1}$ for  ${\sf L}$ in \eqref{L}. In \cite{Matsubara-curvilinear}, a parallel calculation is achieved for a class of Euler-Laplace integrals that are more general than \eqref{IukxI}.
\begin{proposition}
1) The Mellin transform
$M_{u^k x^{\bI},  \tilde \Gamma}$ of the period integral associated to the
simpliciable polynomial $f(x)$ has the following form.
\begin{equation}
M_{u^k x^{\bI},  \tilde \Gamma}= g_{\tilde \Gamma}(z)  \prod_{q=1}^{n+2} \Gamma({\mathcal L}_q(\bI,z,k)), 
\label{MukxIg}
\end{equation}
%\leqno(3.5)$$
 where $g_{\tilde \Gamma}(z)$ is a polynomial in $e^{\frac{2 \pi i z} {\gamma}}$ with
$\gamma = n!vol(\Delta(f)).$ The  function ${\mathcal L}_q(\bI,z,k), q \in [1; n+2]$
linear in $(\bI,z,k)$ with coefficients in $\frac{1}{\gamma} \bZ$ is given by 
\begin{equation}
{\mathcal L}_q(\bI,z,k) = ^t (\bI, z, k)  w_q=\frac{< v_q,\bI> - B_q z+
C_q k}{\gamma}, 
\label{Lq}
\end{equation} 
%\leqno(3.6)$$ 
where $ w_q$
is the $q-$th column vector of the matrix $({\sf L})^{-1}$ for  ${\sf L}$ in \eqref{L}.
\par
2) The $n+2$ linear functions ${\mathcal L}_q(\bI,z,k)$
are classified into the following  three groups.
For $q=n+2,$ we set
\begin{equation}
{\mathcal L}_{n+2}(\bI,z,k) = \frac{\gamma}{\gamma}z = z.
\label{Ln+2}
\end{equation}
% \leqno(3.7)_1$$  

There exists an unique index $q $ such that $ w_q=  ( v_q, - B_q, \gamma)/ \gamma$
for some $ v_q \in \bZ^{n}$ and $B_q  >0. $ We fix such $q$ to be $n+1.$
\begin{equation}
{\mathcal L}_{n+1}(\bI,z,k) =
\frac{< v_{n+1},\bI> - B_{n+1} z }{\gamma} +k .
\label{Ln+1}
\end{equation}
 %\leqno(3.7)_2$$ 

For $q \in [1;n] $ such that $ w_q=  ( v_q, - B_q, 0)/ \gamma$
for some $ v_q \in \bZ^{n},$ and $B_q > 0, $
\begin{equation}
{\mathcal L}_q(\bI,z,k) =\frac{< v_q,\bI> - B_q z }{\gamma}.
\label{LqIz}
\end{equation}
%\leqno(3.7)_3$$ 
%We use the convention so  that $ q \in \{1, \cdots, n\}.$
For these vectors, we have the following equalities
\begin{equation}
\sum_{q=1}^{n+1} w_q =(0,-1,1),  \;\;\; \sum_{q=1}^{n+1} v_q =0.
\label{Wqvq}
\end{equation}
%Here the case $(3.7)_3$ corresponds to such $q$ that $dim\;\tau_q <M-1.$
\par
\label{prop31}
\end{proposition}

\begin{proof}
1) The definition of the $\Gamma-$ function can be formulated as follows
$$ \int_{\bar {\bR}_+} e^{-T}T^\sigma \frac{dT}{T} = (-1+ e^{2 \pi i
\sigma})\int_{\bR_+} e^{-T}T^\sigma  \frac{dT}{T} = (-1+ e^{2 \pi
i \sigma}) \Gamma(\sigma),$$ for the unique nontrivial  cycle
$\bar {\bR}_+$  turning once in an anticlockwise manner around $T=0$ that begins and returns to
$\Re T \rightarrow + \infty.$  In the theory of Bessel functions, 
$\bar {\bR}_+$ is called Hankel cycle. See  \cite{Matsubara-curvilinear} for the use of higher dimensional Hankel cycles in the theory of A-GKZ hypergeometric functions. 

 Let us  consider an action on the chain $C_a=\bar {\bR}_+$
or ${\bR}_+ $ on the complex $T_a-$plane 
$\lambda:C_a \rightarrow \lambda(C_a)$, $a \in [1;n+2],$ 
defined by the relation
$$ \int_{\lambda(C_a)} e^{-T_a}T_a^{\sigma_a} \frac{dT_a}{T_a}
= \int_{(C_a)} e^{-T_a}(e^{2\pi  \sqrt -1 }T_a)^{\sigma_a} 
\frac{dT_a}{T_a}.$$
In particular $\lambda( {\bR}_+) =  \bR_+ +\bar {\bR}_+.$
In terms of this action, the chain ${\sf L}_\ast (
\tilde \Gamma) $ turns out to be homologous to a linear combination with integer coefficients of chains
\begin{equation}
\prod_{q=1 }^{n+1}\lambda^{j_{q}}(\bar {\bR}_+ )  \lambda^{j_{n+2}}(\bR_+) \;\;or \;\;
\prod_{q=1 }^{n+2}\lambda^{j_{q}}(\bar {\bR}_+ ) ,
\label{chains}
\end{equation} 
with $j_q \in \bZ.$  The only singular points of the integrand is at $T_a =0 \;{\rm or}\; \infty, a \in [1;n+2].$ Thus the non trivial integral \eqref{ MukxIz} shall be a   linear combination of integrals along chains \eqref{chains}. 
This  explains the presence of the factor $g_{\tilde \Gamma}(z)$ in 
(\ref{MukxIg}) that is a polynomial in  the exponential functions
$ e^{2 \pi \sqrt -1 j_q{\mathcal L}_q({\bI, \bz, k})},$  $q \in [1; n+2].$ 
We apply the action $\lambda$ to cycles defining the integral (\ref{ MukxIz})
and get $(\ref{MukxIg}).$

Cramer's formula explains the origin of the coefficient $B_q$  in (\ref{Lq}) that is an $n \times n$ minor of 
$\sf L$ \eqref{L}.

The point 2) is reduced to the linear algebra based on Lemma \ref{Bvector}. 
\end{proof}
\par

\begin{cor}

 The Newton polyhedron admits the following representation by the aid of linear functions defined in 
(\ref{Ln+1}), (\ref{LqIz})
\begin{equation}
\Delta (f) = \{\beta \in \bR^n; 0 \leq {\mathcal L}_q(\beta ,0,1) \leq 1 \}   
\label{Deltafb}
\end{equation}
%\leqno(3.10)$$
for $q \in [1 ; n+1].  $
\label{cor32}
\end{cor}
\begin{proof}

After the definition of vectors $ v_{1}, \cdots,
  v_{n+1},$ we can argue as follows. 

For a vector    $\vec i$ on the hyperplane
$\bigl< 0,    \alpha(1) , \lvup , \alpha(n) \bigr>,$ $q \in [1 ; n],$
the scalar product $\bigl< v_{q}, \vec i \bigr> $ vanishes while
  $\bigl< v_{q},  \alpha (q) \bigr>= \gamma. $

For a vector    $\vec i$  from the hyperplane
$\bigl<   \alpha(1) , \cdots , \alpha(n) \bigr>$ not passing through the origin, we have
scalar products $\bigl< v_{n+1}, \vec i \bigr>  =  \bigl< v_{n+1}, \alpha (q) \bigr>  = -\gamma$, $q  \in [1 ; n].$
\end{proof}

\begin{cor}
A monomial $u^\ell x^{\bI} \in \bC[u, x^{\pm}]$ belongs to $S_\Delta$ if and only if the following $n-$tuple of inequalities are satisfied,
$$  0 \leq {\mathcal L}_q(\bI ,0,\ell) \leq 1 $$
for $q  \in [1 ; n].$
\label{SDL}
\end{cor}

\begin{cor}
Under the above situation, the Mellin inverse of
$M_{u^kx^{\bI}, \delta}(s)$ with properly chosen
periodic entire function $g(z)$ with period $\gamma$ gives
(\ref{IukxI}) as follows
\begin{equation}
I_{u^kx^{\bI}, \t \delta} (s)
= \int_{\check \Pi} g(z) \Gamma( z)
\prod_{q=1}^{n+1}\Gamma \bigl({\mathcal L}_q (\bI , z,k )\bigr) s^{-z} dz. 
%\label{IukxI}
\end{equation}
%\leqno(3.11)$$
Here the integration path  $ \check \Pi $ enclosing all poles of $\Gamma( z):$ $ \bZ_{\leq 0}$
has the initial (resp. terminal ) asymptotic direction $ e^{- (\pi/2 -\epsilon ) i}$ (resp. $ e^{ (\pi/2 - \epsilon ) i}$) for some small $\epsilon$.
This  Mellin-Barnes integral 
defines a convergent analytic function in $ -\pi <arg\; s <\pi, $ $0 < |s|< \eta,$ for some $\eta >0.$
\label{cor:IukxI}
\end{cor}
\begin{proof}
In applying the Stirling's formula
$$ \Gamma(z+1) \sim  (2\pi z)^{\frac 1 2} z^z e^{-z},\;\; \Re \; z \rightarrow +\infty,$$
to the integrand of (\ref{IukxI}), we take into account the relation
(\ref{sumBq}).
Here we remind us of the formula $\Gamma(z) \Gamma(1-z)= \frac{\pi}{sin\;\pi z}.$
As for the choice of the periodic function $g(z)$ one makes
use of N\"orlund's technique \cite{Nor}.
In this way we can choose such $g(z)$
that the integrand is of exponential decay on  $\check \Pi.$ Theorem on the Mellin inverse transform \cite[\S 2.14]{Nor}  states that the Mellin-Barnes integral (\ref{IukxI}) for properly chosen $g(z)$ recovers the integral $I_{u^kx^{\bI}, \t \delta} (s). $ 
\end{proof}

In general,  it is a difficult task to find concrete periodic function $g(z)$ that corresponds to 
$I_{u^kx^{\bI}, \t \delta} (s)$ for a cycle $\delta \in H_{n-1} (Z_f).$ The question how to choose $g(z)$ is a {\it desideratum } in the study of period integrals by means of Mellin transforms.
S-J.Matsubara-Heo makes a proposal to establish a correspondence between Pochhammer type cycles and $\Gamma-$series solutions to A-HG equation \cite[section 5]{Matsubara-curvilinear}.
 
\par
\begin{example} 
{\rm Let us illustrate  the above procedures
by a simple example.
$$f(x)= x_1^3x_2^{-1}+ x_1^3x_2^3 + s x_1^2 x_2+1.$$

$${\sf L}= \left
[\begin {array}{ccccc}
3& -1& 0&1\\
3& 3& 0&1\\
0& 0& 0&1\\
2& 1& 1 &1\\
\end {array}\right ],$$
$$({\sf L})^{-1}= \frac{1}{12}
\left[
\begin{array}{cccc}
 3 & 1 & -4 & 0 \\
 -3 & 3 & 0 & 0 \\
 -3 & -5 & -4 & 12 \\
 0 & 0 & 12 & 0
\end{array}
\right], \;\; \gamma =det ( {\sf L})= 12.$$
We have $$ {\mathcal L}_1(\bI,z,k)=\frac{i_1 +3i_2 -5 z}{12},
 {\mathcal L}_2(\bI,z,k)=\frac{3i_1-3i_2 -3z }{12},$$
$$ {\mathcal L}_3(\bI,z,k)=\frac{-4i_1 -4 z}{12} +k, {\mathcal L}_4(\bI,z)=\frac{12z}{12}.$$ Let us denote by
$ \alpha(1) =(3,3),$$ \alpha(2) =(3,-1),$$ \alpha(3) =(0,0),$ $
\alpha(4) =(2,1).$ Then we have
$$ B_1= vol(\tau_1)=2! vol(\alpha(2),\alpha(3),\alpha(4))=5.$$
Similarly $B_2=vol(\tau_2)=3,$ $B_3= vol(\tau_3)=4.$  

It is worth noticing that $h.c.f. \bB=1.$
Thus we have $\gamma = |\bB| =2! vol (\Delta(f)) =12.$

We can look at the  base representatives of $R_f^+$ with the following support points:
$$ \{(i_1,0,1)_{i_1=1}^3, (i_2,1,1)_{i_2=1}^3,(i_3,2,1)_{i_3=2}^3, (4,1,2), (i_3,2,2)_{i_3=4}^5 \}.$$
We have $dim (R_f^+)=11, $ $R_f \cong \bC \oplus R_f^+. $

Later we see (Proposition \ref{prop61}) that the set of vectors in $(\frac{1}{12} \bZ)^3$ given by $ ({\mathcal L}_1(\bI,0,\ell),$ ${\mathcal L}_2(\bI,0,\ell),$ $ {\mathcal L}_3(\bI,0,\ell)),$ $ (\ell,\bI)$ of the above list of support points of base elements of $R_f^+$ coincides with 
$  (\frac{B_1 k}{\gamma}, \frac{B_2k}{\gamma},\frac{B_3 k}{\gamma}  ) =(\frac{5k}{12}, \frac{3k}{12},\frac{4k}{12}  ),$
$ k \in [1; 11]$ modulo $\bZ^3.$}

%\begin{figure}[h]
%\begin{center}
%\includegraphics[width=1.0\textwidth]{Fig2.png}
%\caption{Quotient ring   ${R_f^+}$ }
%\end{center}
%\label{quotientring}
%\end{figure}

\begin{figure}
	\centering
	\includegraphics[width=15cm]{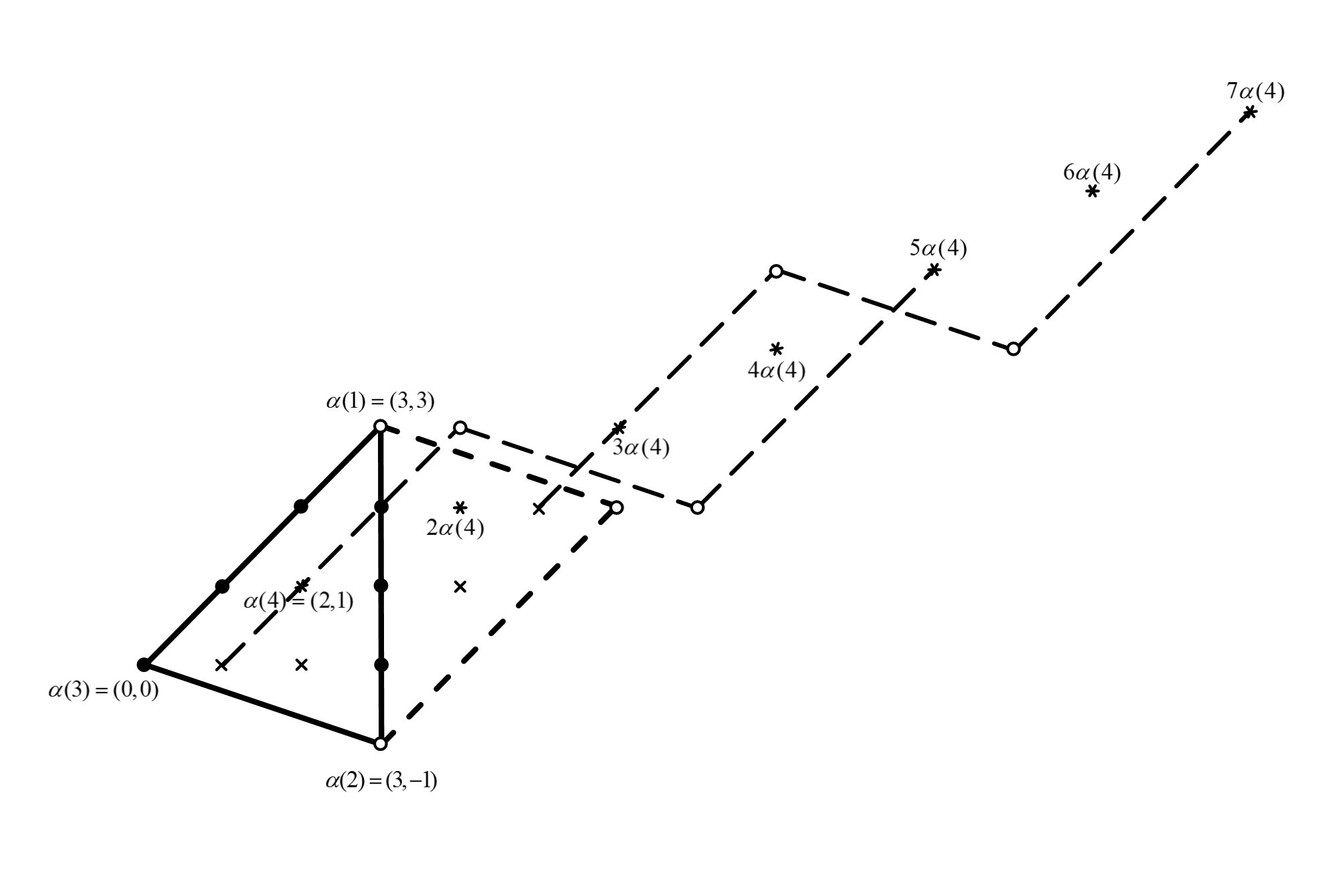}
%\caption{%\leavevmode\\
\begin{minipage}{\linewidth}
		\begin{align*}
		\bar{\lambda}(u^3 x^{3\alpha(4)}) & =\bar{\lambda}(ux_1^3)\\
     \bar{\lambda}(u^5 x^{5\alpha(4)}) & =\bar{\lambda}(ux_1)\\
     \bar{\lambda}(u^7 x^{7\alpha(4)}) & =\bar{\lambda}(u^2 x_1^5 x_2^2)\\
     R_f: \bullet, \times \text{\:\:\: \:\:}&\text{\: \:\:\: \:\:} \ast: k\alpha(4)
		\end{align*} \end{minipage}
%}
		\label{fig:quotientring}
\end{figure}

\label{example345}
\end{example}

\newpage
\begin{example} 
{\rm Now we consider the following (Laurent) polynomial in three variables.
$$f(x)= x_1 x_2 x_3(  x_1+ x_2+ x_3 + s +( x_1 x_2 x_3)^{-1}).$$

$${\sf L}= \left[
\begin{array}{ccccc}
 2 & 1 & 1 & 0 & 1 \\
 1 & 2 & 1 & 0 & 1 \\
 1 & 1 & 2 & 0 & 1 \\
 1 & 1 & 1 & 1 & 1 \\
 0 & 0 & 0 & 0 & 1
\end{array}
\right]$$

$$({\sf L})^{-1}= \frac{1}{4}
\left[
\begin{array}{ccccc}
 3 & -1 & -1 & 0 & -1 \\
 -1 & 3 & -1 & 0 & -1 \\
 -1 & -1 & 3 & 0 & -1 \\
 -1 & -1 & -1 & 4 & -1 \\
 0 & 0 & 0 & 0 & 4
\end{array}
\right], \;\; \gamma =det ( {\sf L})= 4.$$

We have $ {\mathcal L}_1(\bI,z,k)=\frac{3i_1 -i_2 -i_3- z}{4},$ $ {\mathcal L}_2(\bI,z,k)=\frac{-i_1 +3i_2 -i_3- z}{4},$ $ {\mathcal L}_3(\bI,z,k)=\frac{-i_1 -i_2 +3i_3- z}{4}$
$$ {\mathcal L}_4(\bI,z,k)=\frac{-i_1 -i_2 -i_3- z}{4} +k,  \;\;\;\; {\mathcal L}_5 (\bI,z,k)=\frac{4 z}{4}.$$
In this case, we have $B_1=\cdots = B_4=1$ and $R_f^+\cong \oplus_{k=1}^3 \bC ( x_1 x_2 x_3)^k, $  $R_f \cong \bC \oplus R_f^+.$}
\end{example}

%%%%%%%%%%%%%%%%%%%%%%%%%%%%%%%%%%%%%%%%%%%%%%%%%%%%%%%%%
{
\center{\section{Oscillating integrals}\label{oscillating}}
}
Assume $\tilde f(x) =f_0(x) +1$ such that $supp(f_0) \not \ni \{0\} \in \Delta(f_0)$ and $\tilde f$ be a $\Delta(f_0)-$regular polynomial, $Z_{\tilde f}$ non singular.
In this situation we consider the deformation $Z_{f_0+s}$ of $Z_f.$ For generic value of $s \in \bC,$
 a smooth afffine variety $Z_{f_0+s}$ is topologically equivalent to  $Z_f.$

In choosing the coefficients of $f_0$ in a generic position, we may assume that the critical points of $f_0$ i.e. those of $\tilde f$  are of Morse type singularities $c_1, \cdots, c_\gamma$ with $\gamma = n! vol(\Delta(f_0)) $. We construct Lefschetz thimble associated to each critical point $c_j$ as follows.

\begin{dfn} For a fixed complex number $u \in \bC^{\times}$ and  $j \in [1 ; \gamma],$ we consider a path $T^{-}_j$ on $\bC_s$ that starts from $s_j=-f_0(c_j)$ and $\Re\;  (us) \rightarrow  +\infty.$ 
For a 1-parameter deformation of a vanishing cycle $\delta_{j} \in H_{n-1}(Z_{f_0+s})$ with $\delta_{j}$ vanishing at $s_j,$ the cycle $\Gamma_j := \{(s, \delta_{j}); s \in T^{-}_j\}$ of the relative homology $ H_n(\bT^{n}, \Re\; ( uf_0 ) >0; \bZ) $ is called  Lefschetz thimble associated to $\delta_j.$ 
\label{Lefschetz}
\end{dfn}
The set $\mathcal U \subset \bC$ of  generic values of $u$ is defined  by the condition 
\begin{equation}
Arg\; u \not = - Arg (s_{i}- s_{j}) \pm \frac{\pi}{2} 
\label{genericarg}
\end{equation}
%\leqno(2.1)$$ 
for all distinct critical values $s_{i} \not = s_{j} .$ The open set $\mathcal U$ consists of open sectors (fans) within the limiting half-lines of the form   $Arg\; u  = - Arg (s_{i}- s_{j}) \pm \frac{\pi}{2}.$
It is known that for a generic value of $u$ the Lefschetz thimbles $\{ \Gamma_1,\cdots, \Gamma_\gamma \}$ form a basis of the relative homology group $ H_n(\bT^{n}, \Re\; (uf_0 ) > >0; \bZ),$ (see \cite[1.5]{pham83}, \cite[(4.4)]{pham85}). 

Now we introduce the oscillating integral with the phase $f_0(x)$ in the following way,
\begin{equation}
J_{g,\Gamma}(s,u) = \int_\Gamma e^{-u f_0(x)} g(s,x,u) \omega_0
\label{Jdgsu}
\end{equation}  
for a Laurent polynomial $g(s,x,u)$ $\in$ $\bC[s, u,  x_1^{\pm},\cdots, x_n^{\pm}],$ the volume form of the torus $\omega_0=\frac{dx}{x^\b1}$ and $\Gamma$  $\in$ $ H_n(\bT^{n}, \Re\; (uf_0 ) > >0; \bZ).$
Let $C_{i,j}$ be an oriented curve that presents a union of two non-compact non self-intersecting
curves, each of which are located inside of a sector of $\mathcal U$ near infinity. For example, we can take a union $C_{i,j}=C^+_{i,j} -C^-_{i,j}$ of  a curve $C^+_{i,j}$ with the asymptote
 $Arg \;u  =- Arg (s_{i}- s_{j}) + \frac{\pi}{2 }+ \epsilon$ and a curve $C^-_{i,j}$ with the asymptote 
 $Arg \;u  = - Arg (s_{i}- s_{j}) + \frac{\pi}{2 }- \epsilon.$
For such a curve $C = C_{i,j}$ and $\Gamma  \in H_n(\bT^{n}, \Re\; (uf_0 ) > >0; \bZ)$ associated to $ \delta \in H_n(\bT^{n}, \Re\; (uf_0 ) > >0; \bZ)$, we can define the Laplace transform of $ J_{g, \Gamma}(u)$ 
\begin{equation}
 L_{C}(J_{g, \Gamma})(s) = \int_C e^{-us} J_{g, \Gamma}(s, u) \frac{du}{u} = \int_{C} (\int_\Gamma e^{-u (f_0(x)+s)} g(s,x,u) \omega_0 ) \frac{du}{u}
\label{LCjdgs}
\end{equation}
  %\leqno(2.2)'$$
if $\Re \;u (f_0+s)|_{\delta} >0$ near the boundary of $C \times \delta$ at infinity. We recall the notation $\omega_0 = \frac{dx}{x^\b1}.$  For example, in the case of  $C=C_{i,j}$ we can choose a Lefschetz thimble $\Gamma \in \bZ \Gamma_{i} + \bZ \Gamma_{j} + \sum_{\ell}  \bZ \Gamma_{\ell}$ where the summation is taken over $ \Gamma_{\ell}$ such that $Arg\;s_{\ell} \in [Arg\;s_{i}, Arg\;s_{j}].$
The Laplace transform $L_{C_{i,j}}(J_{g, \Gamma})(s)$ is well defined in a sector $\{s \in \bC; -\pi+\epsilon +Arg\; (s_{i}- s_{j}) < Arg \;s  < -\epsilon + Arg\; (s_{i}- s_{j}), \epsilon >0\}.$

%\begin{remark}
%{\rm 
It is not simple to consider $L_{C_{i,j}}(J_{g, \Gamma})(s)$  outside the given sector, even though it is possible to extend analytically the Laplace transform to an open subset of $\bC.$ This requires a detailed study of the asymptotic behaviour of  $J_{g, \Gamma}(s, u)$ from which \cite{Baty1} preferred to abstain. 

In \cite{BJL}, the authors derived properties of the oscillating integrals from the period integrals
by means of Laplace transform (\ref{Cetaint}). In particular they established a relation between the Stokes matrix of oscillating integrals and the monodromy of period integrals. This procedure was followed in \cite{Tan04} to verify Dubrovin's conjecture \cite{Dub2}
for the quantum cohomology of projective space.
%\end{remark}

Now we look at a $\bC[s]$ module $S_{\Delta(f_0)}[s]$ with the basis $\{u^kx^{ \alpha}\}, $ $ \alpha \in k \Delta(f_0).$
In this situation thanks to Theorem ~\ref{thm13} the $\bC[s]$ module
\begin{equation}
H_{DR}^{n-1}(s)\cong \frac{S_{\Delta(f_0)}[s]}{{\mathcal D}_u S_{\Delta(f_0)}[s] + \sum_{i=1}^n{\mathcal D}_{x_i}  S_{\Delta(f_0)}[s] } 
\label{HDR}
\end{equation}
%\leqno(2.3)$$
represents the $(n-1)$th relative de Rham cohomology group  $H^{n-1}(Z_{f_0+s})$ (compare with \cite[\S 11]{Baty1}).
Here  the operators ${\mathcal D}_u,$ ${\mathcal D}_{x_i},  i \in [1;n]$ are defined for $f=f_0+s$ as in Theorem ~\ref{thm13}.

The following is a straightforward consequence of the Stokes' theorem and the fact that the integrand function $e^{-u(f_0(x)+s)} g(s,x,u) =0$ at the infinity boundary of $\delta \times C$.

\begin{lem}\label{lem42}
For $h(s,x,u) \in {\mathcal D}_u S_{\Delta(f_0)}[s] + \sum_{i=1}^n{\mathcal D}_{x_i}  S_{\Delta(f_0)}[s],$ the Laplace transform of the oscillating integral 
$L_C(J_{h, \Gamma})(s)$ vanishes   identically. 
\end{lem}

The vanishing of the Laplace transform for $s$ in an open set means that of $J_{h,\Gamma}(s,u)$ itself. 
Thus for the fixed value $s=1$, $\tilde f(x) =f_0(x)+1,$ the following oscillating integral vanishes 
$$\int_\Gamma e^{-u \tilde f(x)}{\tilde  g}(x,u) \omega_0 $$
for ${\tilde g}(x,u) \in {\cal D}_u S_{\Delta(f_0)}+ \sum_{i=1}^n{\cal D}_{x_i}  S_{\Delta(f_0)}.$

In summary, we arrived at the following conclusion:

\begin{itemize}
 \item
 The ring  (\ref{RF+})  is well adapted to study non-trivial oscillating integrals  like
\begin{equation}
J_{{\tilde g}, \Gamma}(u) =\int_\Gamma e^{-u \tilde f(x)} {\tilde g}(x,u) \omega_0 
\label{JgGamma}
\end{equation}
and their Laplace transforms.
\item
The ring   (\ref{HDR}) is well adapted to study non-trivial oscillating integrals  $J_{g, \Gamma}(s,u)$  (\ref{Jdgsu})
and their  their Laplace transforms.
\end{itemize}

The Brieskorn lattice $G_0$ defined in \cite[2.c]{DS} gives a proper space to examine the oscillating integrals
(\ref{Jdgsu}) for a fixed $u$ and $s \in \bC.$ It would correspond to  $\frac{S_\Delta^{+}}{\sum_{i=1}^n{\mathcal D}_{x_i}  S_\Delta }$ after our notation. 
As the parameter $"u"$ is fixed to a "Planck constant" in \cite{DS}, there is no room to consider the Laplace transform (\ref{LCjdgs}).
In fact, by means of an argument similar to \cite[Proposition 5.2]{TaZ}, it is possible to show that the integrals \eqref{JgGamma} represent non-trivial functions in the parameter $"u"$ (and in $"s"$ for \eqref{Jdgsu} with the integrand from \eqref{HDR}) restricted on a properly chosen simply connected open set. In other words, we have the following: 

\begin{lem}\label{vanishingoscil}
The condition in Lemma \ref{lem42} is a necessary and sufficient for the vanishing of the integral $L_C(J_{h, \Gamma})(s)$, \eqref{LCjdgs}. 
\end{lem}

For $\eta \in \bT$ such that $\gamma \cdot arg\; \eta \not \equiv 0 (mod \; 2 \pi),$ we consider the integration path
$C(\eta)$ that shall be taken as the contour from $\infty$ along a parallel to the direction $arg\; s = arg \; \eta$ sufficiently far away on the left, turning around all the singular loci (\ref{Sloci}) in the anticlockwise sense and back to $\infty$ along a parallel to the direction $arg\; s = arg \; \eta$ sufficiently far away on the right \cite[Theorem 2]{BJL}. For such a path $C(\eta)$ and $ \Re ( u \eta) \geq 0,$ the Laplace transform of $I_{g, \t \delta} (s)$  
(\ref{Ihdelta}) can be defined as follows
\begin{equation}
\int _{C(\eta)} e^{us} I_{g, \t \delta} (s) ds  
\label{Cetaint}
\end{equation}
that is equal to an oscillating integral  $J_{g, \Gamma}(u)$
for a Lefschetz thimble $\Gamma$ associated to the vanishing cycle $\delta$ (see \cite[3.3]{pham85}).
In (\ref{Cetaint}), the path $C(\eta)$ can be homotopically deformed into a union of paths turning around the singular points that correspond to the cycle $\delta.$

Though $C(\eta)$  can be deformed into a zero chain inside the relative homology group  $H_1(\bT, \infty; \Z)$, such a deformation is prohibited for (\ref{Cetaint}). In fact, in the course of a deformation of $C(\eta)$   into a zero chain, the integral would  diverge due to the same reason as explained to define the curve $C_{i,j}$ in (\ref{LCjdgs}).
The inverse Laplace transform (\ref{Cetaint})  shall be defined for $C(\eta)$  with a single asymptotic direction  $\eta \in \bT,$ $\gamma \cdot arg\; \eta \not \equiv 0 (mod \; 2 \pi) $ tending to the infinity.
This consideration suggests that $J_{g, \Gamma_j}(u)= J_{g, \Gamma_j}(0,u),$ $j \in [1; \gamma],$ 
form a $\C$ vector space of dimension $\gamma$ and not of dimension $\bar \gamma = W_{n-1}(H^{n-1}(Z_f))$
discussed in Proposition \ref{prop61}, Remark \ref{WHnZ}. See also the remark after
Proposition \ref{Mfiltration}.

Now we assume $f_0(x)$ to be a Laurent polynomial like in (\ref{f0}).
As a consequence of Proposition \ref{prop31} and Corollary 
\ref{cor:IukxI}, we get the HG equation (see \eqref{RkI} below also) for 
$I_{u, \t \delta} (s)=I_{u^1 x^{0}, \t \delta} (s),$ 
$$R_{(1, 0)} (s, \vartheta_s) I_{u, \t \delta} (s) =0,$$
with
\begin{equation}
R_{(1, 0)} (s, \vartheta_s) = (-\vartheta_s)_\gamma - s^\gamma \prod_{q=1}^{n+1} (\frac{B_q \vartheta_s}{\gamma})_{B_q}
\label{R10}
\end{equation}
where
\begin{equation}
(\alpha)_m = \alpha(\alpha+1) \cdots (\alpha+m-1),
\label{pochhammer}
\end{equation} the Pochhammer symbol.
On  applying the integration by parts, we establish the differential equation with irregular singularities at $u=\infty$
 for $J_{u^2, \Gamma}(u)$ that we denoted by $J_{u^2, \Gamma}(0,u)$ after notation
 \eqref{Jdgsu}:

\begin{equation}
 [u^\gamma -  \prod_{q=1}^{n+1} (\frac{- B_q \vartheta_u}{\gamma})_{B_q}] J_{u^2, \Gamma}(u)=0
\label{Jug}
\end{equation}
for every $\Gamma \in H_n(\bT^n, \Re (u f_0) >>0; \bZ). $
By means of the change of variables $e^{t_1}=(zu)^\gamma$ for the quantization parameter $z$, the equation  (\ref{Jug}) is transformed into
\begin{equation}
 [e^{t_1} -  z^\gamma \prod_{q=1}^{n+1} (- B_q \frac{\partial}{\partial t_1})_{B_q}] \tilde J(t_1, z)=0
\label{Jt1}
\end{equation}
that coincides with the equation for the $J$ function of  the weighted projective space ${\mathbb P}_\bB$ \cite [Corollary 1.8]{CCLT}, \cite[(5.1)]{TaU13}, \cite[Remark 4.3.4 (3)]{DouaiMann}.
Thus we can further develop arguments related to Stokes phenomena of solutions to (\ref{Jt1}) in following \cite{TaU13}. 
See \S \ref{WHP} Weighted  projective space $\mathP_{\bB}$.

%%%%%%%%%%%%%%%%%%%%%%%%%%%%%%%%%%%%%%%%%%%%%%%%%%%%%%%%%%%%%%%%%
{
\center{\section{
Filtration of  period integrals
}\label{filtration}}
}

Now we can state the relationship between the Hodge structure of
the $PH^{n}(\bT^n \setminus Z_{{f}})$ (\ref{PHn}) and the poles of the Mellin transform (\ref{MukxIg}) .

We recall the notation: under the situation described in \S1, the mixed Hodge structure of $PH^{n}(\bT^n \setminus Z_{{f}})$ 
is defined as follows:
$$   Gr_F^{p} Gr^w_{q}PH^{n}(\bT^n \setminus Z_{{f}}) = \frac{(F^p \cap W_q) + W_{q-1}}{ (F^{p+1} \cap W_q) + W_{q-1} }.$$

\begin{thm}
1) Let $u^k x^\bI \in R_f^+$ be a monomial representative such that
$\rho^+( u^k x^\bI ) \in Gr_F^{n-k} Gr^w_{n+1}PH^{n}(\bT^n \setminus Z_{{f}}),$
$0 \leq k \leq n.$
Then the following inequalities hold 
$$0 <{\mathcal L}_q(\bI,0,k) < 1 $$
for $q \in  [1; n+1].$
The poles of Mellin transform  (\ref{MukxIg})  located on the positive real axis $\bR_{>0}$
are included in the infinite set with semi-group structure called poles of positive direction,
\begin{equation}
   \frac{\gamma}{B_q} \left( {\mathcal L}_q(\bI,0,k)  + \bZ_{\geq 0}  \right)\\; q \in [1; n+1],
\label{PosPoles}
\end{equation}
while poles on the negative real axis are included in $\bZ_{\leq 0},$ i.e. each period integral is holomorphic at $s=0.$
\par
2) For a monomial satisfying $\rho^+( u^k x^\bI ) \in Gr_F^{n-k} Gr^w_{n+1+r}PH^{n}(\bT^n \setminus Z_{{f}}),$
$ k \in [0; n], r \in [1 ;  n-1], $
 there exist $r$ indices $ q_1, \cdots, q_{r}, r \in [1;  n+1]$ such that
$   {\mathcal L}_{q_j}(\bI,0,k)=0$ for $j \in [1;r] $ 
but no such  $r+1$ tuple of indices  $ q_1, \cdots, q_{r+1}$ exists.
In other words,  the Mellin transform
\begin{equation}
  \frac{\prod_{q=1}^{n+1} \Gamma({\mathcal L}_q(\bI,z,k))}{\Gamma(1-z)}  
\label{MellinGamma}
\end{equation} 
of the period integral $I_{u^kx^{\bI}, \t \delta} (s)$
has poles of order $r$ at $z=0.$

\label{thm51}
\end{thm}

Proof of the theorem can be achieved by a combination of Theorem~\ref{thm13}
and the Proposition~\ref{prop31}, Corollary ~\ref{cor32}. We remember here that
the $\Gamma(z)$
has simple poles at $z=0,-1, -2, \cdots.$

\vspace{0.5cm}
We shall compare our result of Theorem \ref{thm51}, 1)  with  Kashiwara-Malgrange filtration on the Gauss-Manin system defined by Douai-Sabbah \cite[Theorem 4.5]{DS}. 

First of all, we remark the isomorphism between $S_{\Delta(f)}$  defined for (\ref{fpoly}) and $S_{\Delta(f_0)}$
for (\ref{f0}). In \cite{DS}, the authors studied the Gauss-Manin system associated to a Laurent polynomial whose Newton polyhedron contains the origin as an interior point. To adapt the situation to \cite{DS}, we need to make a transition from $f$ to $f_0.$ 
\begin{equation}
 S_{\Delta(f)} \rightarrow  S_{\Delta(f_0)}, \;\;\; u^kx^\alpha \mapsto u^k x^{\alpha - k \alpha(n+2)}.
\label{isopi}
\end{equation}
The support of   $S_{\Delta(f)}$ (resp.  $S_{\Delta(f_0)}$)  is contained in a cone ${\sf C}\Delta(f):=\sum_{j=1}^{n+1} \bR_{\geq 0}(1, \alpha(j))$  (resp.${\sf C}\Delta(f_0):=\sum_{j=1}^{n+1} \bR_{\geq 0}(1, \alpha_0(j))$ where
$ \alpha_0(j):=\alpha(j)- \alpha(n+2), j \in [1;n+1]$). 
We shall use the notation motivated by the isomorphism (\ref{isopi})
 $$ \pi: (k, \alpha) \mapsto (k, \tilde \pi (\alpha)),$$
with $$ \tilde \pi (\alpha) =  \alpha - k \alpha(n+2) = \alpha  + [ {\mathcal L}_{n+1}(\alpha,0,0)] \alpha(n+2) .$$
Here $[\rho]$ means the maximal integer smaller than $\rho.$ It is worth noticing that $$ u^kx^\alpha \in S_{\Delta(f)} \Leftrightarrow k = -  [ {\mathcal L}_{n+1}(\alpha,0,0)] $$
while for $u^kx^\alpha \in S_{\Delta(f_0)}$ the book keeping index $k$ cannot be recovered from $\alpha$.
% and varies in the range $ 0 \leq k \leq $
The isomorphism (\ref{isopi}) induces an isomorphism between quotient spaces $R_f$
and $R_{f_0}$ defined in (\ref{Rf}). We recall here that  $R_{f_0}$  is obtained from $S_{\Delta(f_0)}$, by taking into account the equivalence relations,
$$ uf_0 \equiv 0,\;\;  (d_j(k,\alpha',f_0) + u x^{\alpha_0(j)}) u^kx^{\alpha' }\equiv 0$$ 
with   $d_j(k,\alpha',f_0)  \in \bQ \setminus \{0\},$ $u^kx^{\alpha'} \in  S_{\Delta(f_0)}$. 
Compare with Lemma \ref{lemma62}.

For the fundamental parallelepiped 
$$  \Pi_{\Delta(f_0)}  =\{\sum_{j=1}^{n+1} t_j \pi (1, \alpha(j)); 0 \leq t_j < 1, j \in [1;n+1]  \}$$
we denote by $Rep(R^+_{f_0})$ the representative polynomials of $R^+_{f_0}$ whose support is located in 
$\Pi_{\Delta(f_0)}.$ It is known that the generating function of $Rep(R^+_{f_0})$ is given by the Ehrhart polynomial $\Psi_{\Delta(f_0)}(t)=$ $\Psi_{\Delta(f)}(t)$ from Definition \ref{Ehrhart}.

Now we introduce the grading on $S_{\Delta(f)}$ as follows
\begin{equation}
{\tilde L}_q (k,\alpha) = \frac{\gamma}{B_q}  {\mathcal L}_{q}(\alpha,0,k)= \frac{<v_q, \alpha> + k \gamma \delta_{n+1,q} }{B_q}
\label{Ltilde}
\end{equation}
that is associated to the poles of positive direction (\ref{PosPoles}).

The grading on  $S_{\Delta(f_0)}$ under the guise of that in \cite[4.a]{DS} can be defined as follows. Let
$$ \Delta_q (f_0)= \bigl< \alpha_0(1), \lvup, \alpha_0(n+1) \bigr>, \;\; q \in [1;n+1]$$
be a $(n-1)$ dimensional simplex face of $\Delta(f_0).$ The $n$ dimensional cone 
 $${\sf C}\Delta_q (f_0):=\bigcup_{\tilde \alpha \in \Delta_q (f_0)} \bR_{\geq 0}(1, \tilde \alpha) $$
is obtained as its coning with the apex at the  origin. Let $L_q(k, \tilde \alpha)$ be a linear function satisfying the following conditions:
$$ L_q(k, \tilde \alpha) =0,  \forall  (k, \tilde \alpha)  \in  {\sf C}\Delta_q,\;\;  L_q(k, 0) =-k,   \forall k \in [1;n]$$
in such a way that
$$  {\sf C}\Delta(f_0)  = \bigcap_{q \in [1;n+1]} \{(r, \beta) \in \bR_{\geq 0} \times \bR^n; L_q(r, \beta) \leq 0\}. $$
Under this situation we have
\begin{lem}
For every $(k, \alpha ) \in \bigcup_{\ell \in \bZ_{\geq 0} } (\ell, \ell \Delta(f)) \subset {\sf C}\Delta(f),$
we have the equality,
$$ L_q( \pi(k,\alpha)) + {\tilde L}_q (k, \alpha)=0, \;\; \forall q \in [1, n+1].$$
\label{LqtildeLq}
\end{lem}

\begin{proof}
For $q \in [1;n],$ the equality below is valid,
$$  L_q( \pi(k,\alpha)) =  L_q(k,\alpha - k \alpha(n+2)) = - \frac{<v_q, \alpha -   k \alpha(n+2) >}{B_q} -k
=  - \frac{<v_q, \alpha  >}{B_q} $$
as $ <v_q,\alpha(n+2) > = B_q$  by (\ref{Bq}).

For $q = n+1$,
 $$L_{n+1}( \pi(k,\alpha)) + {\tilde L}_{n+1} (k, \alpha) = - \frac{<v_{n+1}, \alpha> -   k <v_{n+1},\alpha(n+2) >}{B_{n+1}} -k +  \frac{<v_{n+1}, \alpha  > + k \gamma}{B_{n+1}}.$$ 
We recall that $ v_{n+1}= - \sum_{q=1}^n v_q$ by (\ref{Wqvq}) and see that 
$-<v_{n+1}, \alpha(n+2)> =\gamma -  B_{n+1}.$
\end{proof}

In order to introduce  a filtration defined in terms of Mellin transforms, first we remark that due to (\ref{HDR}), for every $h \in S_{\Delta(f)},$ the following decomposition holds
\begin{equation}
 I_{h, \t \delta} (s) = \sum_{u^\ell x^\alpha \in Rep(R^+_f)} \tilde h_{\ell, \alpha} (s)I_{u^\ell x^\alpha, \t \delta} (s),
\label{Ihdelta}
\end{equation}
 with $ \tilde h_{\ell, \alpha} (s) \in \bC[s].$
Here  $Rep(R^+_{f})$ denotes  representative polynomials of $R^+_{f}$ whose support is located in 
the fundamental parallelepiped $\Pi_{\Delta(f)}$ defined by the relation $\pi(\Pi_{\Delta(f)}) = \Pi_{\Delta(f_0)}$
for the isomorphism $\pi$ arising from \eqref{isopi}.

We introduce a filtration  $\cM_{\beta}$ with $\beta \in \bQ$ on  (\ref{HDR})  with the aid of the notion of poles of positive direction analogous to (\ref{PosPoles}) as follows 
\begin{equation}
\cM_{\beta}:= \{I_{h, \t \delta} (s) ; {\rm minimum\;of\; poles \; of \; positive\; direction\; of} \;  M_{h, \t \delta} (z) \geq \beta\}.
\label{Mbeta}
\end{equation}
In a similar manner we introduce
\begin{equation}
\cM_{>\beta}:= \{I_{h, \t \delta} (s) ; {\rm minimum\;of\; poles \; of \; positive\; direction\; of} \;  M_{h, \t \delta} (z) > \beta\}.
\label{Mbeta1}
\end{equation}
The non-trivial filtration $\cM_{\beta} \not = \emptyset$ means that $\beta \leq \beta_0:= max_{q \in [1;n+1]} \frac{\gamma}{B_q}.$ This filtration is decreasing i.e. $\beta _1 \leq \beta_2 \leq \beta_0 \Rightarrow \cM_{\beta_1} \supset \cM_{\beta_2}.$

 Now we introduce the following rational number defined for $\; h \in S_{\Delta(f)}.$
\begin{equation}
\beta(h):= min_{q \in [1;n+1]} min_{(\ell, \alpha) \in \Pi_{\Delta(f)}, \tilde h_{\ell, \alpha} \not \equiv 0} 
\left( \tilde L_q(\ell, \alpha) - deg \; \tilde h_{\ell, \alpha}\right )
\label{betah}
\end{equation}
Here the notation is the same as in (\ref{Ihdelta}).
From the definition of the Mellin transform (\ref{MukxI}) and the grading 
(\ref{Ltilde}), we see that 
the inequality $\beta(h) \geq \beta$ entails  $I_{h, \t \delta} (s) \in \cM_{\beta}$ for every $\delta(s) \in H_{n-1}(Z_{f}).$ 

In view of Corollary \ref{IukxI}, Theorem \ref{thm51}, we get a result analogous to \cite[Theorem 4.5, Lemma 4.11]{DS}.
\begin{proposition}
The filtration (\ref{Mbeta}), (\ref{Mbeta1})  satisfies the following three properties.

1) We define a positive integer $r(\beta)$ for $\beta \leq \beta_0$
by $$ r(\beta) = max_{u^\ell x^\alpha \in Rep(R_f^+)}  r(\ell, \alpha; \beta), $$
where
$$  r(\ell, \alpha; \beta) =  \sharp \{q \in [1; n+1];  \frac{\gamma}{B_q} {\mathcal L}_q(\alpha,0,\ell)  = \beta \; {\rm for }\; \beta \in [0, \frac{\gamma}{B_q}]  \} .$$
With this notation, we have
$$ (\vartheta_s+\beta)^{r(\beta)} \cM_\beta \subset \cM_{>\beta }.$$
In other words, the action of $\vartheta_s+\beta$ on $\cM_\beta/ \cM_{>\beta}$  is  nilpotent for every $\beta \leq \beta_0$.
We remark also that  $ r(\beta) =  r(\beta-m)$ for $m \in \bZ_{\geq 0}.$

2)   $ \partial_s  \cM_\beta \subset \cM_{\beta+1 }$ for $ \beta \leq \beta_0-1. $

3)   $  s  \cM_\beta \subset \cM_{\beta-1 }$  for $ \beta \leq \beta_0. $
$L_C(J_{h, \Gamma})(s)$
\label{Mfiltration}
\end{proposition}

In \cite[Theorem 4.5]{DS} the Kashiwara-Malgrange filtration is defined on 
the ring $R_f^+$ \eqref{RF+} for $f:$ $\Delta-$regular Laurent polynomial. As it has been remarked in \eqref{JgGamma}, in our setting \eqref{RF+} represents a $\bC[u]$ module of  finite type generated by oscillating integrals. On the other hand, \eqref{Mbeta} is defined for the period integrals associated to cycles of an affine variety $Z_{-uf+1}.$ The oscillating integrals \eqref{JgGamma} are taken along
Lefschetz thimbles that belong to a $\bZ$ module with rank $=\gamma.$ Thus we can state metaphorically that the Kashiwara-Malgrange in  \cite[Theorem 4.5]{DS} is defined for the Laplace transform of our filtration \eqref{Mbeta}. In our previous article \cite[\S 5]{TaU13}, we have paid attention to this correspondence between oscillating integrals and period integrals for an affine variety. 
%%%%%%%%%%%%%%%%%%%%%%%%%%%%%% %%%%%%%%%%%%%%%%%%%%%%%%%%%%%%%%%%%
{
\center{\section{
Hypergeometric group associated to the period integrals
}\label{HGgroup}}
}

For a monomial reperesentative $u^kx^{\bI} \in R_f^+,$ 
we introduce two differential operators of order
$\gamma= n!vol_{n}(\Delta(f))=|\chi(Z_{f})|$ with the aid of the Pochhammer symbol (\ref{pochhammer});
\begin{equation}
P(\vartheta_s)= \prod_{j=0}^{\gamma-1}
\bigl(-\vartheta_{s}\bigr)_{\gamma}
\label{PkI}
\end{equation}
%\leqno(5.1)$$
\begin{equation}
Q_{k, \bI}( \vartheta_s)
= (-1)^\gamma \prod_{q=1}^{n+1}\prod_{j=0}^{B_{q} -1}
\bigl({\mathcal L}_{q}(\bI,-\vartheta_{s},k ) \bigr)_{B_q} . 
\label{QkI}
\end{equation}
with $\vartheta_s =s \frac{\partial }{\partial s}.$
%\leqno(5.2)$$
 We 
have the following theorem as a corollary to the 
Proposition ~\ref{prop31} and Corollary ~\ref{cor32}.

\begin{thm}
The period integral
$I_{u^kx^{\bI},\t \delta} (s)$ is annihilated by the differential operator 
\begin{equation}
R_{(k, \bI)} (s, \vartheta_s)= P(\vartheta_s) -
s^{\gamma}Q_{k,\bI}(\vartheta_s), 
\label{RkI}
\end{equation}
%\leqno(5.3)_1$$
with regular singularities at
\begin{equation}
 \{s \in \mathP^1; (\prod_{q=1}^nB_q^{B_q})  (\frac{s}{\gamma})^\gamma=1, s= 0, \infty\}.
\label{Sloci}
\end{equation}
  In other words,
\begin{equation}
[P(\vartheta_s) -
s^{\gamma}Q_{k,\bI}( \vartheta_s)]I_{u^kx^{\bI},\t \delta} (s)=0. 
\end{equation}
%\leqno(5.4)$$
\end{thm}
It is worthy to remark that  the operator $R_{(k, \bI)}(s,\vartheta_s)$
is a pull-back of the Pochhammer hypergeometric operator of order
$\gamma$ for $\vartheta_t =t \frac{\partial }{\partial t},$
\begin{equation}
\tilde R_{(k, \bI)}(t, \vartheta_t) =P(\gamma \vartheta_t) - t Q_{k,\bI}
(\gamma \vartheta_t), 
\label{RkIt}
\end{equation}
%\leqno(5.3)_2$$
by the Kummer covering  $t=s^{\gamma}.$
In certain cases, the monodromy representation of the kernel of the operator (\ref{RkIt})
turns out to be reducible (\cite[Proposition 2.7]{BH}). To extract the solution subspace with irreducible monodromy from the kernel of  (\ref{RkIt}), we introduce the following
$\gamma-$tuples  of rational numbers contained in $[0, 1)$.
$$ C^+=  \{0, \frac{1} {\gamma}, \cdots,
\frac{(\gamma -1)}{\gamma}\}, $$
\begin{equation}
 C^-(k, \bI)= \bigcup_{q=1}^{n+1}\bigcup_{0 \leq j \leq B_{ q}-1}  
 < \frac{1} {B_{q}} (j+1+ k \delta_{q, n+1} + \frac{< v_{q},\bI>  }{\gamma}) >,
\label{C10}
\end{equation}
where $\delta_{q, n+1} =1$ if $q =n+1$ and $\delta_{q, n+1}=0$ otherwise. For these multi-sets, we define
$$ C^0(k, \bI)= C^+ \cap C^-(k, \bI)$$
as a multi-set intersection (repetitive appearance of the same element for several times is accepted). 
Here we used the notation $  < \rho > = \rho - [\rho] $ that means the fractional part of $\rho \in\bQ .$
The symbol $[\rho]$ stands for the maximal integer smaller than $\rho.$

After {\bf Assumption}  in \S \ref{simplicial}, the Newton polyhedron of $f|_{s=0}$ is the same as that of  $f|_{s \not =0}$
and both of them are $\Delta(f)$ regular. In fact, $Z_f \subset \bT^n$ remains to be smooth for $s=0$ (a Delsarte hypersurface).
This argument justifies the following calculation of $R_f^+$ for the case $s=0.$

%To prepare lemmata, we introduce notations  
%$ A_i(x):= u(c(i) x^{\alpha(j)} -1),$ $i \in \{ 1, \cdots, n\},$  $A_{n+2}(x,s) =   u(p(s) x^{\alpha(n+2)} -1).$
%$B_{j,j+1}(x) := u(d(j) x^{\alpha(j)} - d (j+1)x^{\alpha(j+1)}) , $ $j \in \{ 1, \cdots, n-1\},$ $ B_{n,n+1}(x) = u( d(n) %x^{\alpha(n)} - d(1)x^{\alpha(1)}) .$
\begin{lem}
The denominator of $R_f^+$ in (\ref{RF+}) gives rise to equivalence relations as follows
%{\cal D}_u S_{\Delta(f)}^++ \sum_{i=1}^n{\cal D}_{x_i}  S_{\Delta(f)}^{+}$
%is generated by
$$ (c(k, j) x^{\alpha(j)} -1) u^{k+1} x^\beta \equiv 0, \; (b(k)u -1) u^kx^{\beta} \equiv 0$$
with $u^k x^\beta \in S_\Delta$ and some non-zero constants $b(k), \{ c(k,j) \}_{j=1}^n.$
% and a non zero function $p(s)$ linear in $s.$
\label{lemma62}
\end{lem}

The following is a direct consequence of Lemma \ref{lemma62} in view of Proposition \ref{prop31} and Corollary  
\ref{cor32}.

\begin{lem}
The positive integer $\sharp | C^0(k,\bI)| $ is  well defined on the equivalence class of  $u^k x^ \bI \in R^+_f,$ i.e.
if $u^\ell x^{\bI'} \equiv u^k x^ \bI \in R^+_f$ then $\sharp | C^0(k,\bI)|  = \sharp | C^0(\ell,\bI')|.$
\label{lemma61}
\end{lem}

Thanks to Lemma \ref{lemma61}, we can define a positive integer $\bar \gamma =
\sharp |C^+ \setminus C^0(1, 0)|$ $=
\sharp |C^-(1, 0) \setminus C^0(1, 0)|.$
Then ``the irreducible part''
of the kernel of (\ref{RkIt})  has a monodromy representation  equivalent to that of the kernel of the following Pochhammer hypergeometric operator of order $\bar \gamma$  by virtue of  \cite[Corollary 2.6]{BH}

\begin{equation}
 \bar R_{(1,0)}(t, \vartheta_t)=
\prod_{\alpha^+ \in C^+\setminus C^0(1, 0) }(\vartheta_t+
\alpha^+) - t \prod_{\alpha^- \in C^-(1, 0)\setminus
C^0(1, 0)}(\vartheta_t+ \alpha^-+1). 
\label{R10t}
\end{equation}
%\leqno(5.3)_3$$ 

We consider the set of monomials $S^{mon}_\Delta$ in $S_\Delta.$
We shall further study the  map
$\bar \lambda :S^{mon}_\Delta \rightarrow \frac{1}{\gamma} [0, \gamma)^{n+1} \cap \bQ^{n+1}$
given by 
\begin{equation}\label{lambdauell}
\bar \lambda (u^\ell x^ \bI) = (<{\mathcal L}_{1}(\bI,0,\ell) >, \cdots ,<{\mathcal L}_{n+1}(\bI,0,\ell)>). 
\end{equation}
%In fact $\bar \lambda$ is defined only on the set of monomials $S^{mon}_\Delta$ in $S_\Delta.$
%\leqno(5.5)
We recall here Corollary \ref{SDL} which tells us that $ {\mathcal L}_{q}(\bI,0,\ell)$ being independent of $\ell \in \bZ_{\geq 0}$ equals to its fractional part $< {\mathcal L}_{q}(\bI,0,\ell) >$, $q \in [1; n].$  
By virtue of Theorem \ref{thm51}, we see that 
\begin{equation}
  0 \leq  {\mathcal L}_{n+1}(\bI,0,\ell) <1
\end{equation}
for the monomial representative $u^\ell x^\bI \in R_f^+.$ 
Among these $(n+1)$ tuple of linear functions the relation 
\begin{equation}
\sum_{q=1}^{n+1} {\mathcal L}_{q}(\bI,0,\ell) = \ell
\label{Lql}
\end{equation} holds.

Further we assume that  
\begin{equation}
h.c.f. \bB =1.
\label{gcdB}
\end{equation}
In particular, if one of $B_q$'s is equal to 1, this condition is satisfied.

To examine the map $\bar \lambda$ on $R_f^+,$ we shall further use the isomorphism (\ref{RfRF}) to identify
$R_f$ with $\bC \oplus R_f^+.$

Under this convention, we see that the  map $$\bar \lambda|_{R_f^+}:{R_f^+}  \rightarrow \frac{1}{\gamma} [0, \gamma)^{n+1} \cap \bQ^{n+1} $$ is injective. In fact, from Corollary \ref{cor32}, Lemma \ref{lemma62}  and a formula in Definition \ref{Ehrhart}, it follows that the condition $ \bar \lambda (u^\ell x^ \bI ) \equiv \bar \lambda (u^{\ell'} x^{\bI'} ) \; mod \;\bZ^{n+1}$ yields
$$  u^\ell x^ \bI  \equiv const.  u^{\ell'} x^{\bI'} \;\; in  \; R_{f}^+.$$
for some non-zero constant   $const.$

Being motivated by \cite{CG}, we introduce the following two sets of rational numbers located in  $[0,1)^{n+1}$ ;
$$ Im(\bar  \lambda|_{R_f })^0 = Im(\bar  \lambda|_{R_f } ) \cap ({\bQ^{\times}})^{n+1}. $$

\begin{equation}
(\bZ/\gamma)^0 =\{0, \cdots, \gamma-1\} \setminus \{k \in [0; \gamma-1] ; \gamma \mid  k B_q,  \exists  q \in [1; n+1] \}
\label{Zgamma}
\end{equation}
$$ \cong \left \{   ( <\frac{kB_1}{\gamma} >, \cdots, <\frac{kB_{n+1}}{\gamma}>);  k B_q \not \equiv 0 \; mod \; \gamma,  \;\forall  q \in [1; n+1]  \right \}.$$
We will also make use of the following $\gamma$-tuple of monomials from $S^{mon}_\Delta$
$$ {\mathcal K} = \{ u x^{ \alpha(n+2)}, \cdots,  u^{\gamma} x^{ \gamma\alpha(n+2)} \}. $$

The set  $(\bZ/\gamma)^0$  may seem to be dependent on the choice of $\bB$ or that of $\alpha(n+2)$ according to its definition (\ref{Zgamma}). In fact, it is independent of this choice (See remark \ref{independence}). 
For these sets, we establish the following

\begin{proposition}
 Under the assumption  (\ref{gcdB}),
the order $\bar \gamma $ of the differential operator (\ref{R10t}) is equal to the following quantities.
1) $\sharp  Im(\bar  \lambda|_{R_f } )^0$ , 2)   $ \sharp (\bZ/\gamma)^0 $ , 3)  the dimension of the space
$W_{n-1} (H^{n-1}(Z_f)) , $ i.e.
$\bar \gamma$ is independent of the choice of $\alpha(n+2)$ under the condition (\ref{gcdB}).

 4) Furthermore
$\bar \gamma =
\sharp |C^+ \setminus C^0(\ell, \bI)|$ $=
\sharp |C^-(\ell, \bI) \setminus C^0(\ell, \bI)|$
for every $u^\ell x^\bI \in { R^+_f}.$
\label{prop61}
\end{proposition}
\begin{proof}
 1) First we show the equality
 $\bar \gamma = \sharp (\bZ/\gamma)^0. $ 
%follows from lemma \ref{lemma61}
The definition $\bar \gamma = \sharp |C^+ \setminus C^0(1, 0)|$ entails  
$$ \bar \gamma = \gamma -   \sharp |\bigcup_{q=1}^{n+1}\bigcup_{0 \leq j \leq B_{ q}-1}  
(  \frac{j} {B_{q}} )   \bigcap_{ 0 \leq i \leq \gamma-1} (\frac{i}{\gamma} )|.$$
It is easy to verify that the RHS of the above expression is equal to 
$\sharp (\bZ/\gamma)^0$ under the condition
(\ref{gcdB}).

%In fact the pole of the function $\Gamma({\mathcal L}_{q}(\bI, \gamma z,\ell) )$ are located 
%$$ z= \frac{1}{B_q} (m+ \frac{<v_q, \bI>}{\gamma}),$$ for $m \in \bZ_{\geq 0}.$
%Thus the poles of $ \prod_{q=1}^{n+1}\Gamma({\mathcal L}_{q}(0, \gamma z,1))$ in total form a multi- set of rational points with
%$\bZ_{\leq 0}$ shift $  \cup_{q=1}^{n+1} \{0+  \bZ_{\leq 0}, \cdots,-\frac{ (B_q-1)}{ B_q} $ $+ $ $ \bZ_{\leq 0}\}$ ($\{0 + %\bZ_{\leq 0}\}$  appears $n+1$ times )
%while those of $\Gamma(\gamma z)$ are $\{0+  \bZ_{\leq 0},-\frac{1}{ \gamma }+  \bZ_{\leq 0},\cdots, -\frac{ (\gamma -1)}{ %\gamma } +  \bZ_{\leq 0}\}.$The cardinality of the intersection of these multi-sets is equal to  $\sharp | C^0(1, 0)|.$

2) In order to see  $\sharp  Im(\bar  \lambda|_{R_f } )^0$ 
= $ \sharp (\bZ/\gamma)^0, $ we examine the image of the map
$$\bar \lambda|_{\cal K } :{\cal K} \rightarrow \frac{1}{\gamma} [0, \gamma)^{n+1} \cap \bQ^{n+1}.$$
%To show the second equality we consider the series of integer points $k(1, \alpha(n+2)),$ $  k \in \{0, \cdots, \gamma-1\}. $
%We evaluate
%$ \bar \lambda (k \alpha(n+2),0,k))$
Under the condition (\ref{gcdB}), the map $\bar \lambda|_{\cal K }$ is injective. The injectivity of $\bar \lambda|_{\cal K }$ can be shown in  two steps.  If one of $B_q$ is equal to 1, 
\begin{equation}
  (  \frac{j B_{1}}{\gamma}, \cdots,  \frac{j B_{n+1}} {\gamma}) \equiv 0 \;\;\; mod \;\bZ^{n+1} 
\label{jaBZ}
\end{equation}
if and only if $ j \equiv 0 \; mod \; \gamma.$
In general
suppose that $B_q$ admits the prime decomposition $B_q = \prod_{j=1}^m p_i^{r_i(q)}$ with $r_i(q) \geq 0.$ 
In a similar manner assume $\gamma = \prod_{j=1}^m p_i^{g_i}.$
The minimal number $j $ such that $jB_q$ is divided by $\gamma$ for every $q$ can be expressed as
$$ j =  \prod_{i=1}^m p_i^{g_i - min_q r_i(q)}.$$
The condition  (\ref{gcdB})  means that $  min_q r_i(q) =0. $ Therefore $j = \gamma.$
 This means that $\bar \lambda|_{\cal K }$ is injective. In fact if  (\ref{gcdB}) is not satisfied 
$\bar \lambda|_{\cal K }$ is neither injective nor surjective onto  $ Im(\bar \lambda|_{\cal R_f })^0.$

%rational numbers  are all different because $ u(c(n+2) x^{\alpha(n+2)} -1)$ never belongs to the ideal 
%$J_{f, \Delta}$  for $c(n+2) \in \bC.$
Further we show that for every $k \in [1; \gamma-1],$ we find unique monomial representative $u^\ell x^\bI \in Rep({R^+_f})$ (\ref{Ihdelta}) such that
\begin{equation}  u^\ell x^ \bI  \equiv const.  u^{k} x^{k \alpha(n+2)} \;\; in\;  {R^+_f}
\label{monomialrep}
\end{equation}
for some non-zero constant   $const$ .

By Lemma \ref{lemma62}, for a $\Delta(f)$ regular Laurent polynomial $f$ we have
$${\mathcal L}_{n+1}(k, 0, k \alpha(n+2)) -1 = {\mathcal L}_{n+1}(k-1, 0, k \alpha(n+2) - \alpha(i)) $$
for some $ i \in [1; n]$ satisfying $  u^{k-1} x^{k \alpha(n+2) - \alpha(i)} \in S_\Delta.$
By induction, we find $r \geq 0$ such that
$$ 0 \leq {\mathcal L}_{n+1}(k, 0, k \alpha(n+2)) -r = {\mathcal L}_{n+1} \left(k-r, 0, k \alpha(n+2) - \sum_{i \in {\mathbb  I}_r}\alpha(i) \right)  <1   $$
and $0 < k-r \leq n,$ for an index set $ {\mathbb  I}_r \subset \{1, \cdots, n\}^r. $ Here  $u^{k-r} x^{k \alpha(n+2) - \sum_{i \in {\mathbb  I}_r}\alpha(i)}  \in S_\Delta$ while for every $i' \in \{1, \cdots, n\}$ ,  $u^{k-r-1} x^{k \alpha(n+2) - \sum_{i \in {\mathbb  I}_r}\alpha(i) - \alpha(i')}  \not\in S_\Delta.$ That is to say this descending process starting from  $  u^{k} x^{k \alpha(n+2)}$ stops at certain monomial representative of the quotient ring $R_f^+$ (see figure of Example \ref{example345}).
In this way, the desired monomial   representative  $ u^\ell x^\bI \in  R_f^+$  with $\ell = k - r,$ $\bI =  k \alpha(n+2) - \sum_{i \in {\mathbb  I}_r}   \alpha(i)$ is obtained.
The equality $ \bar \lambda  (  u^\ell x^ \bI  ) $ = $  \bar  \lambda (u^k x^ {k \alpha(n+2) })$ follows immediately from
Corollary \ref{cor32}, Lemma \ref{lemma62}.
This shows that the images of two maps $\bar  \lambda|_{R_f }$ and $\bar  \lambda|_{\mathcal K }$ coincide.
The uniqueness of the required monomial $u^\ell x^ \bI$ follows from the injectivity of $\bar \lambda|_{\cal K }$ and that of $  \bar \lambda|_{R_f }.$ 
%\hl {(this argument shall be modified properly. how?)}
In short $$  Im(\bar  \lambda|_{R_f })^0 = \bigcup_{k \in [1;   \gamma]} \bar \lambda (u^k x^{k \alpha(n+2)}) \bigcap ({\bQ}^{\times})^{n+1}. $$

We see that  $\bar \lambda ( \alpha(n+2),0,1) =  (\frac{B_1}{\gamma}, \cdots, \frac{B_{n+1}}{\gamma})$ 
as the point $ ( 1, \alpha(n+2)) $ represents the weighted barycenter $(1, \sum_{q=1}^{n+1} \frac{B_q \alpha(q)}{\gamma})$ (cf. Proposition \ref{prop31}. 3), Corollary \ref{cor32} and its proof). That is to say,  $Im(\bar  \lambda|_{R_f })^0$ coincides with  the second half of (\ref{Zgamma}).
This shows the equality  $\sharp  Im(\bar  \lambda|_{R_f } )^0$ 
= $ \sharp (\bZ/\gamma)^0.$ 

3)  The image of the map $\bar  \lambda|_{R_f }$ admits the following representation:
 %$$  Im(\bar  \lambda|_{R_f })^0 =
$$ \bigcup_{   u^\ell x^ \bI  \in   { R^+_f } } (\left<\frac{<v_1, \bI>}{\gamma} \right>, \cdots,\left<\frac{<v_{n+1}, \bI>}{\gamma} \right>  ) \bigcap  ({\bQ}^{\times})^{n+1}. $$
This means that only the monomials   $ u^\ell x^ \bI \in I_\Delta^{(1)}$ contribute to 
$ Im(\bar  \lambda|_{R_f })^0 .$ By virtue of Theorem \ref{thm13}, 2) and (\ref{WjPH}), 
we conclude  $  \sharp Im(\bar  \lambda|_{R_f })^0  = dim W_{n-1}(H^{n-1}(Z_f)). $

4)  The argument in 2) shows that
$  \frac{k <v_q, \alpha(n+2)>}{\gamma} =  \frac{k B_q}{\gamma}  $ for every $q \in [1; n+1], k \in \bZ.$
For $u^\ell x^\bI \in {R^+_f}$ such that
$$  u^\ell x^ \bI  \equiv const.  u^{k} x^{k \alpha(n+2)} \;\; in  \; R_{f}^+$$
we calculate
$$ C^-(\ell, \bI)= \bigcup_{q=1}^{n+1}\bigcup_{0 \leq j \leq B_{ q}-1}  
\left< \frac{1} {B_{q}} (j+1+ k \delta_{q, n+1}) + \frac{k}{\gamma} \right>.$$
This set is nothing  but a $\frac{k}{\gamma}$ shift of  $C^-(k, 0)$ introduced in (\ref{C10}). This yields $   \sharp C^0(\ell, \bI) =$    $\sharp C^0(k, 0)$ $=\sharp C^0(1, 0)$
and  $\sharp |C^+\setminus C^0(\ell, \bI)|$ $=
\sharp |C^+ \setminus C^0(1,0)|$ = $\bar \gamma.$  
\end{proof}

\begin{remark}
{\rm We remark here that the space $W_{n-1}(H^{n-1}(Z_f)) \cong W_{n-1}(H^{n-1}(Z_{f_0})) $ of a pure Hodge structure
 is known to be isomorphic to $PH^{n-1}(\bar Z_{f_0})$
where $\bar Z_{f_0}$ is a compactification of  $Z_{f_0}$ defined in a complete simplicial toric variety ${\mathbb P}_{\Delta(f_0)} = Proj S_{\Delta(f_0)}$ (\cite[Proposition 11.6]{BatyCox}).
Thus we have $\bar \gamma = dim (PH^{n-1}({\bar Z}_{f_0})). $

Authors like L.Borisov-R.P.Horja  \cite[Corollary 5.12]{BorHor2},  J.Stienstra \cite{Stienstra} studied the $\bar \gamma$ dimensional space embedded into $H^{n-1}(Z_f)$.
% that can be interpreted as the primitive part of the compact variety ${\bar Z}_{f_0}.$
They investigate this space as solution space of  GKZ A-hypergeometric functions.

The positive integer $\gamma - \bar \gamma$ can be interpreted as the dimension of period integrals of the affine variety $Z_f$ originated from the homology of the ambient space $\bT^n.$ In the exact sequence seen from Theorem \ref{thm13}, 2),
$$ 
0 \rightarrow \bigcup_{2 \leq j \leq n+2 } \rho({\mathcal I}^{(j)}) \rightarrow H^{n}(\bT^n \setminus Z_f) {\rightarrow} \rho({\mathcal I}^{(1)})
\rightarrow 0$$
the term second from left can be interpreted as the dual to "relations" (of rank $\gamma - \bar \gamma$ ) among period integrals.
The quotient by these "relations" would lead to period integrals originated from   $W_{n-1}(H^{n-1}(Z_f)).$ In principle these "relations" can be read off in comparing monodromy representation matrices (\ref{hinf}), (\ref{h0}) in Lemma \ref{levelt} and
 (\ref{eq:hinfty}),  (\ref{eq:h0}) after proper base change of solution spaces  $Ker \; {\tilde R}_{(1,0)}(t, \vartheta_t)  \supset Ker \; {\bar R}_{(1,0)}(t, \vartheta_t)$. 
} 

\label{WHnZ}
\end{remark}

\begin{remark}
{\rm Proposition \ref{prop61} gives a general formula for the hypergeometric equations treated in the following special cases. Here $f(x)$ denotes the Laurent polynomial (\ref{fpoly}).

(1) "Dwork family" \cite{Dominguez} with
$$ f(x)  = x_1^{d_n}+ \cdots +  x_n^{d_n} +1 + s x_1^{w_1} \cdots x_n^{w_n},$$
where $w_0 = d_n- \sum_{j=1}^n w_j,$ $w_j \in \bZ_{>0}, \forall j \in [0;n],$ such that $h.c.f. ({w_0},\cdots,  {w_n})=1.$

(2) Special case of the above "Dwork family" \cite{Dominguez} with $d_n =n$ \cite{Salerno}.
The reduction procedure in the quotient ring achieved in Proposition \ref{prop61}, 2) is a generalised analogy of Algorithm 1 in \cite{Salerno}.
For this case, in our previous publications \cite{Tan04}, \cite{TaU13},  concrete global monodromy representations have been used to establish  homological mirror type statements and Dubrovin's conjecture for the quantum cohomology of (weighted) projective space.

(3) Berglund-H\"ubsch "invertible" polynomial \cite{Gahrs}.
$$ f(x)  = g(1,x) \prod_{i=1}^n x_i^{- E_{i,n}}$$
where
$$ g(x_0, x) = \sum_{j=0}^{n}  \prod_{i=0}^n x_i^{E_{i,j}}$$
a weighted homogeneous polynomial in $(x_0,x)= (x_0, x_1,\cdots, x_n)$ with $weight(x_j)= q_j ,$ $\sum_{i=0}^n q_i E_{i,j} = d, \forall j \in [0;n].$

In our previous preprint \cite{Tan05}, 
we studied  several conditions
to be imposed on a mirror symmetry candidate 
to  the generic multiply weighted homogeneous Calabi-Yau
complete intersection variety (including hypersurface case) defined in the product of the weighted homogeneous projective spaces.
We proposed several properties for a Calabi-Yau complete intersection variety
so that its period integrals can be expressed by means of weighted homogeneous 
weights of its mirror counterpart candidate.
As a corollary, we  remarked certain duality between the monodromy data
and the Poincar\'e polynomials of the Euler characteristic for 
the pairs of these varieties.
}
\label{salerno}
\end{remark}

\begin{remark}
{\rm The arguments developed in the proof of Proposition \ref{prop61}, 2), 3), show that   
${\bar X}_{0}({\mathsf t})$ does not depend on the choice of $\alpha(n+2)$ under the condition (\ref{gcdB}). In other words, the set
 $ (\bZ/\gamma)^0 $ is determined in a way independent of the choice of the deformation term $ s x^{\alpha(n+2)}$ in 
(\ref{fpoly}).}
\label{independence}
\end{remark}

\begin{example}
{\rm We recall the Example \ref{example345}. 
A simple examination of the set of rational vectors
$  (\frac{B_1 k}{\gamma}, \frac{B_2k}{\gamma},\frac{B_3 k}{\gamma}  ) =(\frac{5k}{12}, \frac{3k}{12},\frac{4k}{12}  ),$
$ k \in \{0, \cdots, 11\}$ gives us $(\bZ/12)^0 = \{1,2,5,7,10,11 \}.$
In fact,  there are 6 monomials in $R_f \cap I^{(1)}:$ $\{ u x_1,  u x_1^2,    u x_1x_2,  u^2 x_1^4x_2,  u^2 x_1^4x_2^2,  u^2 x_1^5x_2^2 \}. $}

\end{example}

Let us introduce orderings on the sets of rational numbers $  C^+\setminus C^0(1, 0) $
$$ 0< \alpha^+_1 <   \alpha^+_2 < \cdots <    \alpha^+_{\bar \gamma},$$
and on $  C^-(1,0)\setminus C^0(1, 0), $
$$ 0=\alpha^-_1 \leq   \alpha^-_2 \leq \cdots \leq   \alpha^-_{\bar \gamma}. $$
After \cite[1.2]{CG}, we define the following integer for $ k \in (\bZ/ \gamma)^0,$
\begin{equation}
 p(k) = \sharp \{i ;   \alpha^-_i  < \frac{k}{\gamma} \}  -j
\label{pk}
\end{equation}
where the index $j$ is determined by the relation $ \alpha^+_j  =  \frac{k}{\gamma}. $
We can state the following proposition corresponding to \cite[Proposition 1.5]{CG}.
See also \cite[Proposition 2.5]{DominguezSevenheck}
\begin{proposition}
For $ k \in (\bZ/ \gamma)^0,$ we define the integer $\ell \in [1; n]$
satisfying  (\ref{monomialrep}). Then we have
$$   \sharp \; p^{-1}(n-\ell)= h^{\ell, n+1 - \ell} (PH^n (\bT^n \setminus Z_f)) = dim\;  Gr^{\ell}_F Gr_{n+1}^{w} (PH^n (\bT^n \setminus Z_f)). $$
\label{hln}
\end{proposition}

\begin{proof}
We present our proof in view of typographical errors in \cite[1.2]{CG}.
By definition 
$$ p(k) =\sharp |  (C^-(1,0)\setminus C^0(1, 0))\cap [0, \frac{k}{\gamma}) | - \sharp |  (C^+ \setminus C^0(1, 0))\cap [0, \frac{k}{\gamma}) |$$
$$ = \sharp |  C^-(1,0) \cap [0, \frac{k}{\gamma}) |  - \sharp | C^+ \cap [0, \frac{k}{\gamma}) | =  \sum_{q=1}^{n+1} (1+ [\frac{kB_q}{\gamma}]) -(k+1).$$
By (\ref{sumBq}), this is equal to 
$$ n- \sum_{q=1}^{n+1} (\frac{kB_q}{\gamma}- [\frac{kB_q}{\gamma}]).$$
By (\ref{Lql}),   we have
$  \sum_{q=1}^{n+1} \left< \frac{kB_q}{\gamma} \right> = \ell$
for  $u^\ell x^\bI$ satisfying (\ref{monomialrep}). Taking Theorem \ref{thm51}, 1) into account, we obtain the desired result.
\end{proof}

From the proof of Proposition \ref{prop61}, 2), we see that
$\bar \lambda ( \alpha(n+2),0,1) =  (\frac{B_1}{\gamma}, \cdots, \frac{B_{n+1}}{\gamma})$ 
and the following consequence can be derived from Proposition \ref{hln}.

\begin{cor}
We consider a Mellin transform (up to constant multiplication of variable $t$)  of solutions to (\ref{R10t}),
$$ 
M^0_{u x^{\alpha(n+2)}} (\gamma z) = \frac{ \prod_{q=1}^{n+1} \Gamma(\frac{B_q( 1-\gamma z)}{\gamma})}{\Gamma(1-\gamma z)} $$
by choosing a suitable $g(z)$ in (\ref{MukxIg}).
Then the integer (\ref{pk})  for $ k \in (\bZ/ \gamma)^0$ can be expressed by
$$  p(k) = - \sum_{ \frac{1}{\gamma} \leq z_i  \leq \frac{k+1}{\gamma} }  Res_{z=z_i} d log  M^0_{u x^{\alpha(n+2)}} (\gamma z), $$
where the residues are taken at the set of poles $\{ z_i \}_{i=1}^{m(k)}$
with $$m(k) = \sharp |  \left((C^+  \cup C^-(1, 0)) \setminus C^0(1, 0)  \right) \cap [0, \frac{k}{\gamma})|.$$

\label{logM}
\end{cor}
This means that we can read off the Hodge filtration of  $W_{n+1}(H^{n}(\bT^n \setminus Z_f)$ (equivalently  that of $W_{n-1}( H^{n-1}(Z_f))$ or $PH^{n-1}(\bar Z_f) $ ) from the poles of 
$ \frac{d log M^0_{u x^{\alpha(n+2)}} (\gamma z)}{dz}.$

\par Further we examine the kernel of the operator 
${\bar R}_{(1,0)}(t, \vartheta_t),$ (\ref{R10t})
and monodromy actions on it.
We define  characteristic polynomials
of the local monodromy of $Ker\; {\bar R}_{(1,0)}(t, \vartheta_t)$
at $t=\infty$ (an anticlockwise turn around $t=\infty$ )
\begin{equation}
{\bar X}_{\infty}({\mathsf t})=
\prod_{\alpha^- \in C^-(1,0) \setminus C^0(1,0)}
({\mathsf t}-e^{2\pi \sqrt -1 \alpha^-}) = \frac{\prod_{q=1}^{n+1}
({\mathsf t}^{B_q}-1) }{\varphi( {\mathsf t})  }.
\label{Xinf}
\end{equation}
at $t=0$ (an anticlockwise turn around  $t=0$)
\begin{equation}
{\bar X}_{0}({\mathsf t})=
\prod_{\alpha^+ \in C^+\setminus C^0(1,0)}
({\mathsf t}-e^{2\pi \sqrt -1 \alpha^+})  = \frac{({\mathsf t}^{\gamma}-1)}{\varphi( {\mathsf t})  } .
\label{X0}
\end{equation}
Here  \begin{equation}
 \varphi( {\mathsf t}) = h.c.f. ( \prod_{q=1}^{n+1}
({\mathsf t}^{B_q}-1) , ({\mathsf t}^{\gamma}-1)   ),
\label{varphi0}
\end{equation}
a polynomial of degree $\gamma - \bar \gamma.$
Thus we have $deg {\bar X}_{\infty}({\mathsf t})= deg {\bar X}_{0}({\mathsf t})= \bar \gamma$
in view of Proposition \ref{prop61}.

\begin{proposition}
The two degree $\bar \gamma$ polynomials $ {\bar X}_{0}({\mathsf t})$, ${\bar X}_{\infty}({\mathsf t})$ have integer coefficients.
More precisely $$ {\bar X}_{0}({\mathsf t}) \in \bQ(e^{2\pi \sqrt -1 /\gamma})[{\mathsf t}] \cap \bZ[{\mathsf t}], \;\;\; {\bar X}_{\infty}({\mathsf t}) \in \bQ(e^{2\pi \sqrt -1 /B_1}, \cdots , e^{2\pi \sqrt -1 /B_{n+1}})[{\mathsf t}] \cap \bZ[{\mathsf t}].$$
\label{realc}
\end{proposition}

To prove this Proposition, we prepare the following lemma.
\begin{lem}
Consider two cyclotomic polynomials $P_0({\mathsf t})= \prod_{i=1}^p  ({\mathsf t}^{A_i}-1)$ and $Q_0({\mathsf t}) = \prod_{j=1}^q  ({\mathsf t}^{B_j}-1)$ such that the multi-set of roots of $Q_0({\mathsf t})$ is contained in the multi-set of roots of
 $P_0({\mathsf t})$. Then the rational function $P_0({\mathsf t})/Q_0({\mathsf t})$ is a polynomial with integer coefficients.
\label{cyclotomic}
\end{lem}
\begin{proof}
It is clear that $P_0({\mathsf t})/Q_0({\mathsf t})$ is a polynomial.
 We consider the function $P_0({\mathsf t})/Q_0({\mathsf t})$ for $|t| <1$ and expand the denominator into a convergent power series in making use of the relation
$$ \frac{1}{1-{\mathsf t}^{B_i}} = \sum_{m=0}^\infty {\mathsf t}^{m B_i}.$$
The obtained convergent power series expression of  $P_0({\mathsf t})/Q_0({\mathsf t})$ has integer coefficients. But, in fact, it is a polynomial so the power series breaks down within a finite number of terms.
\end{proof}

\begin{proof} (of Proposition ~\ref{realc}) 

The proof is reduced to a precise formula that can be established by induction.

For an ordered set 
\begin{equation}
\bk = \{ q_1 , \cdots ,q_k \} \subset \{1, \cdots ,n+1\} 
\label{bk}
\end{equation}
 we introduce a rational function
$$   \varphi_\bk( {\mathsf t}) =\prod_{r=1}^{|\bk|} \left( \frac{  ( {\mathsf t}^{C^{(r)}_{q_1,\cdots, q_r}}-1)   }  {( {\mathsf t}-1)}\right)^{(-1)^{r-1}}, \;\; 1 \leq k= |\bk| \leq n+1,$$
that is in fact a polynomial from $\bZ[\t]$ due to Lemma \ref{cyclotomic}.
Here  the exponents shall be interpreted as follows: $C^{(1)}_{q} = h.c.f. (B_q, \gamma)$ and ${C^{(r)}_{q_1,\cdots, q_r}} = h.c.f.(C^{(1)}_{q_1},\cdots, C^{(1)}_{q_r} ),  $ for $r=2, \cdots, n+1. $
We shall remark that $ C^{(n+1)}_{1,\cdots, n+1} =1$ by assumption   (\ref{gcdB}).
We then have a formula  for the polynomial $\varphi( {\mathsf t}),$ (\ref{varphi0}) 
\begin{equation}
 \varphi( {\mathsf t}) = ( {\mathsf t}-1) \prod_{\bk} \varphi_\bk( {\mathsf t}).
\label{varphi}
\end{equation}
where the index $\bk$ runs over all ordered sets  (\ref{bk}) such that  $\{\t ; \varphi_\bk( {\mathsf t})=0\}\cap 
\{\t ; \varphi_{\bk'}( {\mathsf t})=0\} = \emptyset \iff \bk \not = \bk'.$
We apply  Lemma \ref{cyclotomic} to (\ref{Xinf}), (\ref{X0}), (\ref{varphi0}),  in taking into account the formula (\ref{varphi}). 

%First we prove that  $ X_{0}({\mathsf t}) \in \bR [{\mathsf t}]$.  Let us denote by
%$$\tilde X_{0}({\mathsf t})=\prod_{\alpha^+ \in \left(C^+\setminus C^0(1,0) \right)\setminus \{1/2\}}({\mathsf t}-e^{2\pi \sqrt -1 \alpha^+}) $$ 
%a polynomial of degree $\bar \gamma$ or $\bar \gamma -1.$  
%For every $r \leq \frac{\sharp |(C^+\setminus C^0(1,0) )\setminus \{1/2\}|}{2} $ and index set ${\mathcal I}(r) \subset \{1, \cdots, \sharp |(C^+\setminus C^0(1,0) )\setminus \{1/2\}| \}$
%such that $\sharp {\mathcal I}(r)=r$ we can find an unique index set   ${\mathcal I'}(r) \subset \{1, \cdots, \sharp |(C^+\setminus C^0(1,0) )\setminus \{1/2\}| \} \setminus {\mathcal I}(r)$, $\sharp {\mathcal I'}(r)=r$ so that
%\begin{equation}
%\sum_{ j \in {\mathcal I}(r)} \alpha_j^+  + \sum_{ j' \in {\mathcal I'}(r)} \alpha_{j'}^+ =r.
%\label{asym}
%\end{equation}
%This shows that    $\tilde X_{0}({\mathsf t}) \in \bR [{\mathsf t}]$ and the desired result follows.

%To prove a similar statement for $X_{\infty}({\mathsf t})$ we remark that for every 
%$   \frac{j}{B_q}$ $\in $ $(C^-(1,0) $ $\setminus C^0(1,0))\setminus \{1/2\},$ $ 1 \leq j \leq B_q-1,$ there exists unique
%$j', $  $ 1 \leq j' \leq B_q-1$ satisfying $j+j' = B_q.$
%Further argument is similar to that for  $ X_{0}({\mathsf t}).$
\end{proof}

For the polynomials introduced in (\ref{Xinf}), (\ref{X0}),
we define two vectors $(\cal A_1, \cal A_2, \cdots, \cal A_{\bar \gamma}),$
$ (\cal B_1, \cal B_2,$ $\cdots,$ $\cal B_{\bar \gamma})$
$\in \bZ^{\bar \gamma}, $
after the following relations:
$${\bar X}_{\infty}({\mathsf t})={\mathsf t}^{\bar \gamma}+
{\cal A}_1{\mathsf t}^{\bar \gamma-1}+
\cal A_2{\mathsf t}^{\bar \gamma-2}+ \cdots + \cal A_{\bar \gamma},$$
$${\bar X}_{0}({\mathsf t})={\mathsf t}^{\bar \gamma}+
\cal B_1{\mathsf t}^{\bar \gamma-1}+
\cal B_2{\mathsf t}^{\bar \gamma-2}+ \cdots + \cal B_{\bar \gamma}.$$

An examination of the elements of $\alpha^- \in C^-(1,0) \setminus C^0(1,0)$ leads us to conclude that  $X_{\infty}({\mathsf t})$ is a product of $({\mathsf t}-1)^{n}$
and a factor vanishing on a set of non-real complex numbers with rational arguments symmetrically located with respect to the real axis.  This means that  $\cal A_{\bar \gamma} =(-1)^n.$
From the symmetry of the  set $(C^+\setminus C^0(1,0) )\setminus \{1/2\}$ with respect to $1/2$ it follows that $\cal B_{\bar \gamma}=1. $
In other words, for every $r \leq \frac{\sharp |(C^+\setminus C^0(1,0) )\setminus \{1/2\}|}{2} $ and index set ${\mathcal I}(r) \subset \{1, \cdots, \sharp |(C^+\setminus C^0(1,0) )\setminus \{1/2\}| \}$
such that $\sharp {\mathcal I}(r)=r,$ we can find an unique index set   ${\mathcal I'}(r) \subset \{1, \cdots, \sharp |(C^+\setminus C^0(1,0) )\setminus \{1/2\}| \} \setminus {\mathcal I}(r)$, $\sharp {\mathcal I'}(r)=r$ so that
\begin{equation}
\sum_{ j \in {\mathcal I}(r)} \alpha_j^+  + \sum_{ j' \in {\mathcal I'}(r)} \alpha_{j'}^+ =r.
\label{asym}
\end{equation}

A theorem due to A.H.M. Levelt (see \cite{Lev}, \cite{BH}) tells us that the global monodromy representation of the solution space
$ Ker \; {\bar R}_{(1,0)}(t, \vartheta_t)$ with irreducible monodromy can be recovered from polynomials  (\ref{Xinf}),(\ref{X0}).

\begin{lem}
The global monodromy group ${\bar H}_{\gamma, \bB}$ of the  solution space
$ Ker \; {\bar R}_{(1,0)}(t, \vartheta_t)$ is generated by
\begin{equation}
 h_\infty=
\left(
\begin{array}{llccll}
0 & 0 & \cdots &0  & (- 1)^{n+1} \\
1 & 0 &  \ddots &0 &-\cal A_{\bar \gamma-1} \\
0 &1 & \ddots  &0&-\cal A_{\bar \gamma-2} \\
\vdots &\ddots  & \ddots &\vdots &\vdots\\
0 & 0 & \cdots &  1 &-\cal A_1 \\
\end{array} \right),
\label{hinf}
\end{equation}
%\leqno(5.8)$$
\begin{equation}
 (h_0)^{-1}=
\left(
\begin{array}{llccll}
0 & 0 & \cdots &0 & -1 \\
1 & 0 &  \ddots &0 &-\cal B_{\bar \gamma-1} \\
0 &1 & \ddots  &0&-\cal B_{\bar \gamma-2} \\
\vdots &\ddots  & \ddots &\vdots &\vdots\\
0 & 0 & \cdots &  1 &-\cal B_1 \\
\end{array} \right).
\label{h0}
\end{equation}
 Here $h_0 \in GL(\bar \gamma, \bZ)$ (resp. $h_\infty \in GL(\bar \gamma, \bZ)$) corresponds to the monodromy action around a loop turning anticlockwise around $t=0$ (resp. $ t=\infty$). The monodromy action around a point $t  = 1$ is given by $h_1 = (h_0 h_\infty)^{-1}.$
\label{levelt}
\end{lem}

\begin{proposition}
For $u^\ell x^\bI \in R^+_f$ such that $\bar \lambda ( u^\ell x^\bI ) = \bar \lambda ( u^k  x^{k \alpha(n+2)})$ by \eqref{lambdauell}, 
\eqref{monomialrep},
we have the following relation between  corresponding  differential operators   (\ref{RkIt}):
$$ \bar R_{(\ell, \bI)} (t, \vartheta_t) = \bar R_{(k, k \alpha(n+2))} (t, \vartheta_t) =  \bar  R_{(1, 0) }(t, \vartheta_t + \frac{k}{\gamma}).$$
The  monodromy representation of the solution space to the equation
$$ \bar R_{(\ell, \bI)} (t, \vartheta_t)  u(t)=0$$
is equivalent to that for $  ker \bar  R_{(1, 0) } (t, \vartheta_t) $  up to exponent shifts
$$ \alpha^+ \rightarrow \alpha^+ + \frac{k}{\gamma}, \;\;\;  \alpha^- \rightarrow \alpha^- + \frac{k}{\gamma}.$$
\label{shift}
\end{proposition}
The proof follows from  the representation of the set $C^-(\ell, \bI)$ in this situation obtained in the proof of Proposition \ref{prop61}, 4).

Let us denote by
$\omega^i, i=0,1,2, \cdots,\gamma -1,$
the non-zero singular points of the equation (\ref{Sloci}).
% i.e.$\{s \in \bC ;{\prod_{q \in I^+}B_q} -
%\bigl({\prod_{\bar q \in I^-}B_{\bar q}}\bigr)s^{\gamma}=0\}.$
\begin{thm}
There is a $\bar \gamma$ dimensional subspace (i.e.  $Ker\; \bar R_{(1,0)}(s^\gamma, \vartheta_s/\gamma)$) of the solution space $ker R_{(1,0)}(s, \vartheta_s)$ ($k=1, \bI=0$ in  (\ref{RkI}) ) 
whose  global monodromy group ${\bar H}_{\gamma, \bB}$ is given by generators 
$$ 
M_{\omega^0} = h_1= (h_0 h_\infty)^{-1}, M_\infty =
h_\infty^{\gamma},
M_{\omega^i} = h_\infty^{-i}h_1 h_\infty^i (i=1,2, \cdots,\gamma -1),
$$
for the matrices  $h_0, h_\infty,  h_1$ defined in Lemma \ref{levelt}.
Here $M_{\omega^i}$ denotes the monodromy action around the point
$\omega^i \in  {\mathP}^1_s.$ 

In particular  ${\bar H}_{\gamma, \bB}$ is a discrete subgroup of $GL(\bar \gamma, \bZ)$ and $h_1^2 = id$ for $n:$ odd.
%where $ K_{\gamma,\bB} :$ the maximal real sub-field of the field  $\bQ( e^{2\pi \sqrt -1 /\gamma},e^{2\pi \sqrt -1 /B_1}, \cdots , e^{2\pi \sqrt -1 /B_{n+1}} ).$

\label{thm53}
\end{thm}
\begin{proof} The monodromies of the solutions annihilated by
$\bar R_{(1,0)}(t, \vartheta_t)$
are given by $h_0,$ (resp. $h_1, h_\infty$) after
Lemma \ref{levelt}
at $t=0,$ (resp.$t=1,\infty$).
Let us think of a $\gamma-$leaf covering
$\tilde {\mathP}^1_t$ of ${\mathP}^1_s$
that corresponds to the Kummer covering $s^{\gamma} =t.$

For the solution space $Ker\; {\bar R}_{(1,0)}(s^\gamma, \vartheta_s/\gamma),$
its monodromy can be described as follows.
In lifting up the path around $t=1,$ the first leaf of
$\tilde{\mathP}^1_s,$ the monodromy $h_1$ is sent
to the conjugation with a path around $t=\infty.$ That
is to say, we have $M_{\omega^1}= h_\infty^{-1}h_1 h_\infty.$
For other leaves, the argument is similar (Reidemeister-Schreier method).

The monodromy around $s=0$ would  be $h_0^\gamma,$ but in fact this is an identity matrix
in view of  (\ref{X0}). This fact matches Theorem \ref{thm51}, 1)
stating that all the period integrals (\ref{IukxI})  are holomorphic near $s=0.$

The statement that it is a discrete subgroup of $GL(\bar \gamma, \bZ)$ follows from Proposition \ref{realc}.
\end{proof}

An element $h$ of the monodromy group   ${\bar H}_{\gamma, \bB}$ acts naturally
on the space of ${\bar \gamma} \times {\bar \gamma}$-matrices by
$$
  { X} \mapsto h^T \cdot { X} \cdot \bar h,
$$
where $h^T$ is the transpose of $h$ and $\bar h$ the complex conjugate to $h$ .
The following is a corollary of Proposition \ref{hln} and Theorem \ref{thm53}:
\begin{cor}
There exists a non degenerate Hermitian invariant $\bar \FX$ such that 
$$ h^T \cdot {\bar \FX} \cdot  h ={\bar \FX}$$
for every $h \in  {\bar H}_{\gamma, \bB} \subset GL(\bar \gamma, \bZ).$
%is spanned by a real ${\bar \gamma} \times {\bar \gamma}$ matrix whose 

The signature $(\sigma^+, \sigma^-)$ of $\bar \FX$ is given by 1) $|\sigma^+- \sigma^-| =0$  for $n$ even, 2) 
$$|\sigma^+- \sigma^-|  =\bar \gamma$$
%= h^{\frac{n+1}{2}, \frac{n+1}{2}} + 2 \sum_{\ell =0}^{ \frac{n-1}{2}} (-1)^\ell h^{\ell, n+1-\ell },$$
that is the index of the variety $\tau (\bar Z_f)$ for $n$ odd. 
%\hl{ Whether it is possible to calculate the right hand side.}
\label{signature}
\end{cor}

\begin{proof}

To see the existence of  a non degenerate Hermitian invariant $\bar \FX$ we apply  \cite[Theorem 4.3]{BH} to our situation. 
It would be enough to recall  the condition (\ref{asym}) for  $ {\bar X}_{0}({\mathsf t})$ and an analogous symmetry condition
for the roots of  $ {\bar X}_{\infty}({\mathsf t}).$
It is also possible to repeat the argument \cite[\S 3]{TaU13} that can be applied to our situation almost verbatim. 

In combining Proposition \ref{hln} formulated for the value  (\ref{pk})
and \cite[Theorem 4.5]{BH}, we see that the signature $(\sigma^+, \sigma^-)$ of the generating quadratic invariant is given by
$$  |\sigma^+- \sigma^-| =  \sum_{\ell =1}^{ n}  (-1)^\ell h^{\ell, n+1-\ell },$$
while $h^{n+1,0} = h^{0, n+1}=0,$ \cite[Proposition 5.3]{Baty1}.
The symmetry of Hodge numbers $h^{p,q} = h^{q,p}$ establishes the result 1).

As for result 2), it follows from  the definition of the index $\tau (\bar Z_f).$
In this case (the special eigenvalue of $h_1 =-1$), $\bar H_{\gamma,\bB}$ is  isomorphic to a subgroup of $O(\bar \gamma, \bC)$ 
up to  complex conjugate isomorphism. This result can be obtained 
by means of an application of \cite[Proposition 6.1]{BH} to our situation.
Again $\bar H_{\gamma,\bB}$ is self dual in the sense of  \cite{BH}, thanks to the condition (\ref{asym}) for  $ {\bar X}_{0}({\mathsf t})$ and an analogous symmetry condition for the roots of  $ {\bar X}_{\infty}({\t}).$
\end{proof}

\begin{remark}
 {\rm Corollary \ref{signature} means that,  for $n:$even  (the special eigenvalue of $h_1 =1$), the monodromy group $\bar H_{\gamma,\bB}$ is isomorphic to a subgroup of $O(\frac{\bar \gamma }{2}, \frac{\bar \gamma}{2} )$ up to real conjugate isomorphism. In other words, it is  isomorphic to a subgroup of $Sp(\frac{\bar \gamma }{2}, \bC)$ 
up to a complex conjugate isomorphism. }
\end{remark}

{
\center{\section{Weighted projective space $ {\mathbb P}_\bB$}\label{WHP}}
}

Let $\bN$ be the dual lattice to $\bM$ introduced in (\ref{Fpoly}).
We define the polar polyhedron $\Delta^o(f_0) \subset \bN_\R$ of the Newton polyhedron 
$\Delta(f_0)$
$$ \Delta^o(f_0) = \{v \in \bN_ \R; <v, \alpha> \geq -1,   \forall \alpha \in \Delta(f_0) \}.$$
\begin{lem}
We denote by $\FA_j$ the $j-$th column vector of the upper $n \times (n+2)$ part of the matrix $\sf L^{-1}$ inverse to  
$\sf L$    (\ref{L}).
 The polar polyhedron $\Delta^o(f_0)$ is represented as the convex hull of vectors
$$\{ \frac{\gamma \FA_1}{B_1}, \cdots,  \frac{\gamma \FA_{n+1}}{B_{n+1}}\}. $$
\label{polar}
\end{lem}
The lemma can be seen from the following relations that hold for every $j \in [1;n+1]$,
$$ <\alpha_0(i), \frac{\gamma \FA_j}{B_j}> =  -1 + \frac{\gamma}{B_j} \delta_{i,j}, \forall i \in [1; n], \;\; <\alpha_0(n+1), \frac{\gamma \FA_j}{B_j}> =  -1+ \frac{\gamma}{B_{n+1}}.$$

The normal fan of $\Delta^o(f_0)$ is generated by cones over the proper faces of $\Delta(f_0)$ \cite[Lemma 3.2.1]{CK}.
Every cone of the interior point fan of $\Delta(f_0)$ is generated by $(n-k)-$tuples of $ \{\alpha_0(j)\}_{j=1}^{n+1}$ satisfying 
\begin{equation} \sum_{j=1}^{n+1} B_j \alpha_0(j) =0
\label{barycentre}
\end{equation} for $k \in [0;n-1].$ In fact, we have seen that $ \alpha(n+2)= \sum_{j=1}^{n+1} \frac{B_j \alpha(j)}{\gamma}$ during the proof of Proposition  \ref{prop61}, 2).

This means that the toric variety $\mathP_{\Delta^o(f_0)}$ is nothing but the weighted projective space $\mathP_{\bB}$ under the condition (\ref{gcdB}). 
It is known that the toric variety $\mathP_{\bB}$ is a Fano variety if and only if $ \frac{\gamma }{B_j} \in \Z, \forall j \in [1;n+1],$ \cite[Lemma 3.5.6]{CK}.
 Lemma \ref{polar} yields that if the $\Delta^o(f_0)$ is an integral polyhedron then the weighted projective space $\mathP_{\bB}$ is a Fano variety.
This means that  $\Delta(f_0)$  is a reflexive polytope. See \cite[Theorem 4.1.9]{Baty2}.

Now we examine the relation between the Stokes matrix for the oscillating integral (\ref{Jdgsu}), (\ref{Jug})
and the Gram matrix of the full exceptional collection on $\mathP_\bB.$
 
First we recall that the monodromy group $H_{\gamma, \bB}$ in   $GL(\gamma, \Z)$ of  $Ker \; \tilde R_{1,0} (t, \vartheta_t)$
(\ref{RkIt}) is
generated by two matrices (\cite[Theorem 1.1]{TaU13})
\begin{equation} \label{eq:hinfty} %\label{eq:h0}  \label{eq:b}
H_\infty
   = \begin{pmatrix}
      0 & 0 & \dots & 0 & (-1)^n \\
      1 & 0 & \dots & 0 & -\FB_{\gamma-1} \\
      0 & 1 & \dots & 0 & - \FB_{\gamma-2} \\
      \vdots & \vdots & \ddots & \vdots & \vdots \\
      0 & 0 & \dots & 1 & - \FB_1
     \end{pmatrix}  
\end{equation}
and
\begin{equation} \label{eq:h0}
 H_0^{-1}
   = \begin{pmatrix}
      0 & 0 & \dots & 0 & 1 \\
      1 & 0 & \dots & 0 & 0 \\
      0 & 1 & \dots & 0 & 0 \\
      \vdots & \vdots & \ddots & \vdots & \vdots \\
      0 & 0 & \dots & 1 & 0
     \end{pmatrix},
%\end{align*}
\end{equation}
where
\begin{equation} \label{eq:b}
 \prod_{q=1}^{n+1} (\t^{B_q} - 1)
  = \t^{\gamma} + \FB_1 \t^{\gamma-1} + \FB_2 \t^{\gamma-2} + \dots + (-1)^{n+1}
\end{equation}
is the characteristic polynomial
of the monodromy 
%$Ker \; \tilde R_{1,0} (t, \vartheta_t)$
at infinity. It is worth noticing that $H_{\gamma, \bB}$ admits a reducible monodromy representation 
and the Levelt type theorem \cite[Theorem 3.5]{BH} cannot be directly applied to $\tilde R_{1,0} (t, \vartheta_t).$
%The solution basis of $Ker \tilde R_{1,0} (t, \vartheta_t)$ used to obtain  (\ref{eq:hinfty}),  (\ref{eq:0})  
The validity of the  monodromy representation (\ref{eq:hinfty}), (\ref{eq:h0}) is based on the existence of a vector $v$ that is cyclic with respect to $H_0$
satisfying
\begin{align} \label{eq:condition}
 H_0^i v = H_\infty^{-i} v, \qquad i \in [1; \gamma-1].
\end{align}
See \cite[Proposition 2.3]{TaU13}.

Let $(\mE_i)_{i=1}^\gamma$ be the full strong exceptional collection
on $D^b \coh \mathP_{\bB}$
given as
$$
 (\mE_1, \dots, \mE_\gamma) = (\mO, \dots, \mO(\gamma-1)),
$$
and $(\mF_1, \dots, \mF_\gamma)$ be
its right dual exceptional collection
characterised by the condition
$$
 \Ext^k(\mE_{\gamma-i+1}, \mF_j) = 
  \begin{cases}
   \C & i=j, \text{ and } k=0, \\
   0 & \text{otherwise}.
  \end{cases}
$$
In other words
$$
 \chi(\mE_{\gamma-i+1}, \mF_j)
  = \delta_{ij}
$$
where
\begin{equation} \label{eq:Euler}
 \chi(\mE, \mF)
  = \sum_k (-1)^k \dim \Ext^k (\mE, \mF)
\end{equation}
is the Euler form.
Note that
$
 \mF_1 = \mO_{\bP_{\bB}}(-1)[n]
$
and
$
 \mF_\gamma = \mE_1 = \mO_{\bP_{\bB}}.
$

We construct a hypersurface $Y$ of weighted degree $\gamma = |\bB|$ in $\mathP_{\bB}$
by means of a "transposition" of the Newton polyhedron $\Delta(f_0).$ 
First we consider a  $n \times (n+1)$ matrix defined by columns $\alpha_0(j) = \alpha(j) - \alpha(n+2), j \in [1;n+1],$
for (\ref{fpoly}) whose rows we denote by $\Fb(i), i \in [1;n]$
\begin{equation} [\alpha_0(1), \cdots, \alpha_0(n+1)] = \left [\begin {array}{c} 
\Fb(1)\\
\vdots\\
\Fb(n)\\
\end {array}\right ].
\label{Fb}
\end{equation}
The polynomial with generic coefficients below, constructed from (\ref{Fb}),  is weighted homogenous  with respect to the weight system
$w(y_q) = B_q, q \in [1;n+1]$ with weight $|\bB| = \gamma$
\begin{equation}
 f^T(y) = y^{\b1}(\sum_{i=1}^{n} b_{i} y^{\Fb(i)} + b_{n+1}). 
\label{fT}
\end{equation}
This can be seen from the relations $<\Fb(i), \bB> =0, \forall  i \in [1;n],$
$<\b1, \bB> =\gamma.$

Let $Y \subset {\mathbb P}_\bB$ a hypersurface defined by a weighted polynomial of weight $\gamma = |\bB|$
$$ Y = \{ y \in \mathP_{\bB};  f^T(y) =0\}.$$
If it is smooth, then it is a Calabi-Yau manifold. 
%\begin{proposition}
Under the standard definition of the Poincar\'e polynomial  $P_{Y}({\mathsf t})$ of a weighted homogenous hypersurface  (\cite[3.4]{Dolg}) 
the following equality is established;
\begin{equation}
P_{Y}({\mathsf t})  = \frac{(1 - {\mathsf t}^\gamma )}{ \prod_{q=1}^{n+1} (1- \mathsf t ^{B_q})}
= (-1)^n  \frac{X_{\infty}({\mathsf t})}{X_{0}({\mathsf t})} =  (-1)^n  \frac{\bar X_{\infty}({\mathsf t})}{\bar X_{0}({\mathsf t})}.
\label{poincare}
\end{equation}

In considering the derived restrictions
$
\{ {\bar \mF}_i \}_{i=1}^\gamma
$
of $\{ \mF_i \}_{i=1}^\gamma$ to $Y,$
that split-generate the derived category $D^b \coh Y$
of coherent sheaves on $Y$, we see 
that the Stokes matrix for (\ref{Jug}), (\ref{Jt1}) is given by
$$
 S_{ij}
=(\sigma_i, \sigma_j)
  = \chi({\bar \mF}_i, {\bar \mF}_j)
  = \chi(\mF_i, \mF_j)
     + (-1)^{n-1} \chi(\mF_j, \mF_i)
  = \chi(\mF_i, \mF_j)
$$
for $i < j,$ $S_{ii}=1$ and $S_{ij}=0$ for $i>j.$ These numbers are given as intersection number of vanishing cycles
used to define Lefschetz thimbles in Definition \ref{Lefschetz}.
The $\gamma \times \gamma$ matrix
\begin{equation}
 \FX = \left(\chi({\bar \mF}_i, {\bar \mF}_j) \right)_{i,j=1}^\gamma   = \left(S_{ij} + (-1)^{n-1}S_{ji} \right)_{i,j=1}^\gamma  
\label{FX}
\end{equation}
corresponds to the Hermitian invariant of the monodromy group $H_{\gamma, \bB}$ (\cite[Proposition 4.1]{TaU13}) satisfying,
$$ h^T \cdot { \FX} \cdot  {\bar h} ={ \FX}$$
for every $h \in  { H}_{\gamma, \bB} \subset GL( \gamma, \bZ).$
We recall that we can assume $h = \bar h  \in  { H}_{\gamma, \bB}$ by virtue of   (\ref{eq:hinfty}), (\ref{eq:h0}),  (\ref{eq:b}).
The space of Hermitian invariants of   ${H}_{\gamma, \bB}$
is one dimensional and generated by (\ref{FX}).

In summary,  by applying 
\cite[Theorem 5.1]{TaU13} to our situation, we get the following.
%by shifts, cones and direct summands
\begin{thm}
1) The Stokes matrix $(S_{ij})_{i, j=1}^\gamma$ for the quantum cohomology
of the weighted projective space  $\mathP_{\bB}$ is equivalent to that for the oscillating integral (\ref{Jug}).
 
2) This Stokes matrix
is given by the Gram matrix
of the full exceptional collection $(\mF_i)_{i=1}^\gamma$ on $\mathP_\bB$
with respect to the Euler form;
\begin{equation} \label{eq:stokes}
 S_{ij} =  \chi( \mF_i, \mF_j ).
\end{equation}
\label{Dubrovin}
\end{thm}

This generalises a conjecture proposed by Dubrovin \cite{Dub2} for Fano manifolds 
that has been proven first by D.Guzzetti \cite{Guz} for the case of the projective space
(see \cite{Tan04} also).
 H.Iritani \cite[Remark 4.13]{Iritani_ISQCMSTO} mentions  the correspondence between Lefschetz thimbles $\Gamma_i$ and exceptional collection of coherent sheaves $\mF_i$ for the case of a weighted projective space. This is a consequence of the assertion that there exist $\mG_1, \cdots, \mG_\gamma$ in the Grothendieck group $K(\mathP_\bB)$  such that
$\chi(\mG_i, \mG_j ) = S_{ij}$ \cite[Theorem 4.11, Corollary 4.12]{Iritani_ISQCMSTO}. 
Thus the above Theorem \ref{Dubrovin} can be considered as a concrete realisation of Iritani's theorem.
In \cite{Iritani_ISQCMSTO}, however, we find no precise description of $(\mG_i)_{i=1}^\gamma$ except the case $\mathP_{\bB} = \mathP^n.$

In view of the formulation of Gamma conjectures \cite[Definition 4.6.1]{GammaConj} it would be desirable to
give a  newly adapted version of the above Theorem \ref{Dubrovin}. 
%\label{proppoincare}
%\end{proposition}tic conti

Our further quest in this direction shall be explained
in an article in preparation where we consider the 
analytic continuation of the Mellin-Barnes integral representation of the quantum cohomology, i.e. solution to the differential equation with 
irregular singularities.  
 This can be understood as a trial to develop the argument that the authors of 
 \cite{GammaConj} followed till $(6.2.5)$ in \cite[6.2]{GammaConj}. Afterwards the authors abandoned this way as they reduced the 
 quantum cohomology
 to the oscillating integral, i.e. the monodromy is expressed in terms of Lefschetz thimbles. 
 
 Our new approach in progress uses the Mellin-Barnes integrals that can be recovered from \eqref{R10}, \eqref{Jug}:
 $$ \frac{1}{2 \pi i} \int_{-\epsilon - i\infty}^{-\epsilon + i\infty} u^z \prod_{q=1}^{n+1} \Gamma(-B_q z) dz, \;\; \epsilon >0.  $$
 The monodromy (or Stokes data) are expressed in terms of the analytic continuation of the integral containing product of Gamma factors. 
 
\begin{remark}
The situation explained in this section can be summarised into a diagram as follows
$$\begin{array}{ccc}
\gamma=rk H^{n}(\bT^n \setminus Z_{f_0}),  f_0: LG\; potential & \iff  & \mathP_{\bB}, rk \;K(\mathP_{\bB}) = \gamma\\
\bar \gamma = rk PH^{n-1}(\bar Z_{f_0}),  Delsarte:\bar Z_{f_0} \subset \mathP_{\Delta(f_0)} &\iff&  C.Y.: Y \subset  \mathP_{\bB}, rk \;\iota_\ast K (\mathP_{\bB}) = \bar \gamma.\\
\end{array}$$
For $\iota : Y \hookrightarrow \mathP_{\bB}$  the inclusion
we denoted with $\iota_\ast K (\mathP_{\bB})$ the subgroup of $K(\mathP_{\bB})$
generated by $\{ [\iota_\ast {\mathcal{O}}_Y(i)] \}_{i \in \bZ}$.
The correspondence "$\iff$" indicates mirror symmetry in certain sense (Givental' $I=J$ mirror \cite{CCLT},  \cite{Iritani_ISQCMSTO}, Batyrev dual polytope mirror \cite{Baty2} or expected homological mirror symmetry by M. Kontsevich \cite[11.D.]{Kontsevich_HAMS}). 

\end{remark}
%%%%%%%%%%%%%%%%%%%%%REFERENCES%%%%%%%%%%%%%%%%%%%%%%%%%%%%%%%%%%

\vspace{\fill}
%%%%%%%%%%%%%%%%%%% ADDRESS %%%%%%%%%%%%%%%%

\noindent

\begin{flushleft}
\begin{minipage}[t]{6.2cm}
  \begin{center}
{\footnotesize
Department of Mathematics,\\
Galatasaray University,\\
\c{C}{\i}ra$\rm\breve{g}$an cad. 36,\\
Be\c{s}ikta\c{s}, Istanbul, 34357, Turkey.\\
{\it E-mails}:   {tanabe@gsu.edu.tr}}
\end{center}
\end{minipage}
\end{flushleft}

\end{document}